\documentclass[11pt]{article}

\usepackage{enumitem}
\usepackage{amsmath, amssymb, amsthm}
\usepackage{physics}
\usepackage{graphicx}
\usepackage{subcaption}
\usepackage{xcolor}
\usepackage{tikz}
\usepackage{hyperref}
\usepackage{algorithm, algorithmic}
\usepackage{array}
\usepackage{lmodern}
\usepackage{geometry}
\usepackage{indentfirst}
\usepackage{titlesec}
\usepackage{fancyhdr}
\usepackage{appendix}
\usepackage{tikz-cd}
\newtheorem{theorem}{Theorem}[section]
\newtheorem{lemma}[theorem]{Lemma}

\newtheorem{proposition}[theorem]{Proposition}
\newtheorem{corollary}[theorem]{Corollary}

\newtheoremstyle{remarkitalic} 
  {}{}                 
  {\normalfont}        
  {}                   
  {\itshape}           
  {.}                  
  { }                  
  {}                   

\theoremstyle{remarkitalic}
\newtheorem{remark}[theorem]{Remark}

\titleformat{\section}
  {\Large\bfseries}
  {\thesection}{0.5em}{}

\hypersetup{
    colorlinks = true,
    linkcolor  = red,
    urlcolor   = black,
    citecolor  = blue
}

\begin{document}
\title{The Cauchy problem for logarithmic Schrödinger
equation with a perturbation of power law nonlinearity} 
	\author{Mohamed Bensaid\\Univ. Lille, CNRS, Inria, UMR 8524 - Laboratoire Paul Painlevé, F-59000 Lille, France\\ \url{mohamed.bensaid@univ-lille.fr}}
\maketitle
\begin{abstract}
    In this paper, we study the Cauchy problem for the nonlinear Schr\"odinger equation with a nonlinearity combining a logarithmic term and a local nonlinearity which has polynomial growth. 
Our main result is the construction of solutions in the energy space $W_1(\mathbb{R}^d)$ in the subcritical case. 
We also discuss a sufficient condition for global existence in the $L^2$--critical case in each of these spaces. 
Moreover, we establish a more general Cauchy theory in the space $\Sigma_\alpha$, as introduced without details in \cite{13}.
\end{abstract}
\section{Introduction }
\subsection{Settings}
We consider the nonlinear Schrödinger equation with a nonlinearity combining a logarithmic term and a smooth function.
\begin{equation}\label{NLS}\tag{NLSlog-g}
i\partial_t u + \Delta u + \lambda u \log|u|^2+ ug(|u|^2) = 0,
\qquad x \in \mathbb{R}^d,
\end{equation}
where $d \ge 1$.

This equation is of particular interest because it combines two types of nonlinearities that appear in different physical and mathematical contexts. 
Studying such a combination allows us to extend the known results for purely logarithmic or purely power-type nonlinearities.

In this work, we focus on the Cauchy problem associated with \eqref{NLS}. 
Our main objective is to establish global well-posedness results in suitable functional spaces under appropriate assumptions on the nonlinearity $g$. 
This includes proving the existence, uniqueness, and continuous dependence of solutions on the initial data, as well as the conservation laws.

\subsection{Nonlinear Schrödinger equation with power-type nonlinearity}
We consider the case $\lambda = 0$. Under suitable assumptions on the nonlinearity $g$, the Cauchy problem is locally well-posed in $H^1(\mathbb{R}^d)$ (see \cite{1}).

We first recall the classical case of a pure power nonlinearity, namely $g(s) = \mu s^{\sigma}$, with $0 < \sigma < \frac{2}{d-2}$ if $d \geq 3$, and $0 < \sigma < \infty$ if $d = 1,2$. In this case, the equation reduces to the nonlinear Schrödinger equation of power type
\begin{equation}\label{NLSP}\tag{NLSP}
i\partial_t u + \Delta u + \mu |u|^{2\sigma}u = 0,
\qquad x \in \mathbb{R}^d.
\end{equation}
It originates from quantum mechanics, in particular from the modeling of the propagation
of a laser beam in a nonlinear medium, typically an optical fiber designed to transport
an electromagnetic signal over long distances (see \cite{Has,Agrawal,222}). Moreover, the nonlinear Schrödinger equation possesses several invariances: invariance under spatial 
and temporal translations, Galilean invariance, and invariance under multiplication by a complex constant of modulus $1$. According to Noether's theorem, each invariance gives rise to a conservation law, namely energy, mass, and angular momentum.
$$\mathcal{E}(u)=\frac12\norm{\nabla u}_2^2-\frac{\mu}{2\sigma+2}\int\abs{u}^{2\sigma+2}\mathrm{d}x,\quad M(u)=\frac12\int \abs{u}^2\mathrm{d}x,\quad J(u)=\frac12\Im\int \overline{u}\nabla u\mathrm{d}x$$
\\
When $\mu < 0$, the Cauchy problem is globally well-posed in $H^1(\mathbb{R}^d)$ as a consequence of energy conservation for any $0<\sigma<\frac{2}{d-2}$. By contrast, when $\mu>0$, global well-posedness holds only under the condition
$\sigma<\frac{2}{d}$ (see \cite{1}). Whereas, if $\frac2d\leq\sigma< \frac{2}{d-2}$, finite-time blow-up
may occur in any dimension, as shown by the Virial argument (see \cite{Gla1977}).
The equation $\eqref{NLSP}$ admits special traveling wave solutions called solitons
$$
u(t,x)=e^{i\omega t}Q_\omega(x),
$$
The existence of such a solution, which is unique, positive, radially symmetric, and radially decreasing, has been established in several works \cite{15,lions,Kwong1989,Serrin,1,lecoz2009}.

Let us return to the question of finite-time blow-up.
In the critical case $\sigma = \frac{2}{d}$, the Gagliardo–Nirenberg constant $C_{GN}$ in Proposition \ref{GN} is explicitly given by Weinstein in \cite{W}
$$C_{GN} = \frac{d+2}{d} \norm{Q_1}_{L^2}^{-4/d}.$$
In this case, we have that if $\norm{u_0}_{L^2} < \norm{Q_1}_{L^2}$, then the corresponding solution of \eqref{NLSP} with $\sigma = \frac{2}{d}$ is global in time and disperses as $t \to \pm\infty$. In the critical mass case $\norm{u_0}_{L^2} = \norm{Q_1}_{L^2}$, a fundamental result due to F.~Merle (1992) provides a complete description of finite-time blow-up solutions. More precisely, any solution with minimal mass that blows up in finite time must coincide, up to the symmetries of the flow, with the explicit pseudo-conformal profile $S$.
As a consequence, one obtains a dynamical classification of solutions at the critical mass level: if $u_0 \in H^1$ satisfies $\norm{u_0}_{L^2} = \norm{Q}_{L^2}$, then any corresponding solution is either the minimal blow-up solution (up to symmetries) or is global in time and disperses as $t \to \pm \infty$ (see \cite{M}).
\subsection{Nonlinear Schrödinger equation with logarithmic-type nonlinearity}
In the case $\mu=0$, the equation reduces to the so-called
logarithmic Schrödinger equation
\begin{equation}\label{NLSlog}\tag{logNLS}
i\partial_t u + \Delta u + \lambda u\log|u|^2 = 0,
\qquad x \in \mathbb{R}^d.
\end{equation}
It was introduced by Białynicki-Birula and Mycielski \cite{30} to address a physical requirement
known as \emph{separability of independent systems}.
 This implies that, for any initial data of the form
$$
u_{0} = u_{1,0} \otimes u_{2,0}, \quad \text{i.e.} \quad
u_{0}(x_1,x_2) = u_{1,0}(x_1) u_{2,0}(x_2),
\quad \forall x_1 \in \mathbb{R}^{d_1}, \ \forall x_2 \in \mathbb{R}^{d_2},
$$
the solution $u$ to \eqref{NLSlog} in dimension $d=d_1+d_2$ with initial data $u|_{t=0}=u_{0}$ is
$$
u(t) = u_1(t) \otimes u_2(t),
$$
where $u_j$ is the solution to \eqref{NLSlog} in dimension $d_j$ with initial data $u_{j,0}$, for $j=1,2$.\\
Moreover, \eqref{NLSlog} posses the same invariances as \eqref{NLSP} along with an additional scaling invariance: if $u = u(t,x)$ satisfies \eqref{NLSlog}, 
then for any $\kappa > 0$,
$$
u_\kappa(t,x) = \kappa  u(t,x)  e^{2 i \lambda t \log \kappa},
$$
also satisfies \eqref{NLSlog}.

The Cauchy theory for the logarithmic Schrödinger equation was first established by 
Cazenave and Haraux \cite{29} in the case $\lambda>0$, using the theory of monotone operators, 
in contrast with the power-type case \eqref{NLSP}. The restriction on the sign of $\lambda$ 
stems from the complexity of the energy associated with this equation, given by
$$
\mathcal{E}(u) = \frac{1}{2} \| \nabla u \|_{L^2}^2 - \frac{\lambda}{2} \int |u|^2 \big(\log|u|^2-1\big)  \mathrm{d}x.
 $$
Because of the form of this energy, the Cauchy theory is naturally studied in the energy space which is an Orlicz space
$
W_1(\mathbb{R}^d) 
$ (see Appendix \ref{C}),
for more details concerning Orlicz spaces, see \cite{11}. The problematic part of the potential energy,
$$
 \lambda \int_{\{|u|<1\}} |u|^2 \log|u|^2 \mathrm{d}x,
$$
is counted positively when $\lambda<0$, which makes the problem easier.\\
The case $\lambda < 0$ was treated by Carles and Gallagher \cite{13}, who investigate the long-time behavior of solutions and prove that the dispersion rate is
$
\left(t\sqrt{\log t}\right)^{-d/2},
$
which decays faster than the standard dispersion rate for \eqref{NLSP}, namely
$
t^{-d/2}.
$
In order to handle the lack of suitable estimates associated with this part of the potential energy, we follow \cite{13} and work in the weighted space $\Sigma_\alpha$ (see Appendix \ref{C}).
In \cite{3}, Hayashi and Ozawa revisited and simplified the construction of solutions in the energy space $W_1(\mathbb{R}^d)$, and extended the result to $H^1(\mathbb{R}^d)$.

\subsection{Towards more general nonlinearities ...}
A first generalization of \eqref{NLSlog} is obtained by combining the logarithmic term with a power-type nonlinearity,
leading to the equation \begin{equation}\label{NLSlp}\tag{NLSlogP}
i\partial_t u + \Delta u + \lambda u\log\abs{u}^2+\mu u\abs{u}^{2\sigma} = 0,
\qquad x \in \mathbb{R}^d,\ d\ge 1.
\end{equation}

This equation was introduced by Carles and Gallagher \cite{13} in order to study the large-time behavior of a power-type perturbation of \eqref{NLSlog}.
Since the logarithmic nonlinearity has a strong influence on the dynamics (see \cite{13}),
the relevant comparison should be made with solutions of \eqref{NLSlog}, rather than with those
of \eqref{NLSP}. Due to the different methods used to study the cases of \eqref{NLSP} and \eqref{NLSlog}, 
an additional difficulty arises in the analysis of this equation.
The existence of solutions was established by Carles and Gallagher \cite{13} in the space $\Sigma_\alpha$
for $\lambda,\mu < 0$.
More recently, uniqueness in $H^1(\mathbb{R}^d)$ has been proved by Hayashi in~\cite{14}, without addressing existence.

Inspired by these results, 
our aim is to establish well-posedness not only for \eqref{NLSlp}, but also for \eqref{NLS} in the energy space $W_1(\mathbb{R}^d)$.
\subsection{Main result}
Before stating the main result, let us recall the following energy functional associated with equation \eqref{NLS}:
\begin{equation*}\label{010}
E(u) := \frac{1}{2} \int |\nabla u|^2 \mathrm{d}x 
       + \frac{\lambda}{2} \int |u|^2 \mathrm{d}x 
       - \frac{\lambda}{2} \int |u|^2 \log|u|^2 \mathrm{d}x
       - \int G(|u|^2) \mathrm{d}x,
\end{equation*}
where 
$
G(y) := \frac12\int_0^{y} g(s) \mathrm{d}s 
$ for all $y\in \mathbb{R}_+.$
Inspired by the power-type NLS case, we assume that the function $g$ satisfies the following assumptions:
\begin{enumerate}[label=(A\arabic*), ref=(A\arabic*)]
\item \label{A1}
$g \in \mathcal{C}^0([0,+\infty),\mathbb{R})\cap \mathcal{C}^1((0,+\infty),\mathbb{R})$, 
$g(0)=0$, and
$
\lim_{s\to 0} s g'(s)=0.$
\item \label{A2}
There exist constants $C>0$ and $\sigma>0$ such that
$$
|g'(s)| \le C s^{\sigma-1}, \quad \text{for all } s\geq 1,
$$
where
$$
\begin{cases}
0<\sigma<\dfrac{2}{d-2}, & \text{if } d\ge 3,\\[6pt]
0<\sigma<+\infty, & \text{if } d=1,2,
\end{cases}
$$
\end{enumerate}
The assumptions \ref{A1}-\ref{A2} are introduced to ensure the local existence of solutions. In addition, we distinguish between the focusing and the defocusing cases:
\begin{enumerate}[label=(B\arabic*), ref=(B\arabic*)]
    \item \label{B1} Focusing case: There exists $s_0>0$ such that 
          $G(s_0)>0.$
    \item \label{B2} Defocusing case: For all $s_0>0$, 
          $G(s_0)\leq0.$
\end{enumerate}
 
\begin{theorem}
     Let $\lambda\in \mathbb{R}^*$, assume that \ref{A1} and \ref{A2} are satisfied, and that either \ref{B2} holds with $0 < \sigma < \frac{2}{d-2}$, or \ref{B1} holds with $0 < \sigma < \frac{2}{d}$.
\begin{itemize}
    \item For any $u_0\in H^1(\mathbb{R}^d)$, there exists a unique global solution
$
u \in \mathcal{C}(\mathbb{R}, H^1(\mathbb{R}^d))\cap\mathcal{C}^1(\mathbb{R},H^{-1}_{\rm loc}(\mathbb{R}^d))
$
for \eqref{NLS}, and the mass $M(u(t))$ is conserved.
\item For any $u_0 \in W_1(\mathbb{R}^d)$, there exists a unique global solution
$
u \in \mathcal{C}(\mathbb{R}, W_1(\mathbb{R}^d))\cap\mathcal{C}^1(\mathbb{R},W_1'(\mathbb{R}^d))
$
for \eqref{NLS}, and the mass $M(u(t))$ is conserved. In addition if $\lambda>0$, $
u \in \mathcal{C}_b(\mathbb{R}, W_1(\mathbb{R}^d))\cap\mathcal{C}^1(\mathbb{R},W_1'(\mathbb{R}^d))
$ and the angular momentum $J(u(t))$ and the energy $E(u(t))$ are conserved in time. 
\item For any $u_0 \in \Sigma_\alpha(\mathbb{R}^d)$, there exists a unique global solution
$
u \in \mathcal{C}(\mathbb{R}, \Sigma_\alpha(\mathbb{R}^d))\cap\mathcal{C}^1(\mathbb{R}, \Sigma_\alpha'(\mathbb{R}^d))
$
for \eqref{NLS}. Moreover, the mass $M(u(t))$, the angular momentum $J(u(t))$, and the energy $E(u(t))$ are conserved in time.
\end{itemize}
\end{theorem}
\begin{remark}
    Note that if $u_0 \in W_1(\mathbb{R}^d)$, the conservation of energy is proved only in the case $\lambda > 0$, the same situation holds for \eqref{NLSlog}, see \cite{3}. 
We expect that the conservation law remains valid for $\lambda < 0$, although this case involves additional technical difficulties that will be discussed later. 
On the other hand, if $u_0 \in \Sigma_\alpha(\mathbb{R}^d)$, we can prove energy conservation for any sign of $\lambda$.
\end{remark}
We now turn to a continuity property of the flow with respect to the initial data.
In particular, it allows us to control the evolution of the solutions when the initial data converge 
in $L^2_{\rm loc}(\mathbb{R}^d)$.
\begin{theorem}[Continuity of the flow]\label{08}
    Let $T>0$ and $u_{n,0}, u_0 \in H^1(\mathbb{R}^d)$, and let $u_n(t)$ and $u(t)$ be the solutions corresponding to the initial data $u_{n,0}$ and $u_0$, respectively. Suppose that
$$
u_{n,0} \to u_0 \quad \text{in } L^2_{\rm loc}(\mathbb{R}^d).
$$ 
and that there exist $K_T>0$ such that $$\sup_{\substack{n \geq 0 \\ t \in [0,T]}}\norm{u_n(t)}_{H^1}\leq K_T.$$
Then, for all $t \in [0,T]$, we have 
$$
u_n(t) \rightharpoonup u(t) \quad \text{in } H^1(\mathbb{R}^d)\quad \text{ and }\quad u_n(t) \to u(t) \quad \text{in } L^2_{\rm loc}(\mathbb{R}^d).
$$
\end{theorem}
\begin{remark}
    We remark that if in addition we suppose that $u_n(t)$ is uniformly bounded in $W_1(\mathbb{R}^d)$ we get that $u_n(t)\rightharpoonup u(t)$ in $W_1(\mathbb{R}^d)$.
\end{remark}
The last main result concerns equation \eqref{NLSlp}. 
It is well known that, in the case of the power-type nonlinear Schrödinger equation, when 
$
\sigma = \frac{2}{d},
$
the global behavior of solutions is closely related to the critical mass. We establish a sufficient condition for global existence in the focusing case, expressed in terms of the mass of the ground state for \eqref{NLSP}, rather than for \eqref{NLSlp}.
\begin{theorem}
      Let $\lambda\in \mathbb{R}^*$, $\mu>0$, and $\sigma=\frac2d$.

      \begin{itemize}
    \item For any $u_0\in H^1(\mathbb{R}^d)$ such that $\norm{u_0}_2<\norm{Q}_2$, there exists a unique global solution
$
u \in \mathcal{C}(\mathbb{R}, H^1(\mathbb{R}^d))\cap\mathcal{C}^1(\mathbb{R},H^{-1}_{\rm loc}(\mathbb{R}^d))
$
for \eqref{NLSlp}, and the mass $M(u(t))$ is conserved.
\item For any $u_0 \in W_1(\mathbb{R}^d)$ such that $\norm{u_0}_2<\norm{Q}_2$, there exists a unique global solution
$
u \in \mathcal{C}(\mathbb{R}, W_1(\mathbb{R}^d))\cap\mathcal{C}^1(\mathbb{R},W_1'(\mathbb{R}^d))
$
for \eqref{NLSlp}, and the mass $M(u(t))$ is conserved. In addition if $\lambda>0$, $
u \in \mathcal{C}_b(\mathbb{R}, W_1(\mathbb{R}^d))\cap\mathcal{C}^1(\mathbb{R},W_1'(\mathbb{R}^d))
$
and the angular momentum $J(u(t))$ and the energy $E(u(t))$ are conserved in time.
\item For any $u_0 \in \Sigma_\alpha(\mathbb{R}^d)$ such that $\norm{u_0}_2<\norm{Q}_2$, there exists a unique global solution
$
u \in \mathcal{C}(\mathbb{R}, \Sigma_\alpha(\mathbb{R}^d))\cap\mathcal{C}^1(\mathbb{R}, \Sigma_\alpha'(\mathbb{R}^d))
$
for \eqref{NLSlp}. Moreover, the mass $M(u(t))$, the angular momentum $J(u(t))$, and the energy $E(u(t))$ are conserved in time.
\end{itemize}
\end{theorem}
\begin{remark}
   We emphasize that we only discuss the $L^2$--mass subcritical case. One may naturally ask what happens when $\sigma > \frac{2}{d}$, or even in the critical case $\sigma = \frac{2}{d}$ with $\|u_0\|_{L^2} = \|Q\|_{L^2}$. 
In these situations, it would be necessary to analyze the local Cauchy theory, which is apriori not straightforward in our setting. Moreover, at least formally with a similar Viriel argument, one can show that under certain conditions on $u_0$ and $\sigma$, blow-up may occur.
\end{remark}
\begin{remark}
    The transformation
$$
u(t,x) \longmapsto 
\left|\frac{\lambda}{\mu}\right|^{\frac{1}{2\sigma}}
\exp\left( \frac{i\lambda}{\sigma}
\big( \log|\lambda| - \log|\mu| \big)t \right)
 u\!\left( |\lambda|t, \sqrt{|\lambda|} x \right),
$$
defines an invariance of the equation (\ref{NLSlp}). This invariance allows one to reduce the problem for \eqref{NLSlp} to the normalized case
$\lambda, \mu \in \{\pm 1\}$.
\end{remark}


To simplify the presentation of the Cauchy theory, we restrict our attention to the evolution for positive times.
By reversibility of the equation, the results can be extended to negative times.
For instance, if $ u(t,x) $ is a \emph{strong solution}, then
$ v(t,x) = \overline{u(-t,x)} $ is also a solution of our problem with initial data
$ v_0 = \overline{u_0} $.
This observation allows us to extend the theory to negative times without difficulty.

This paper is organized as follows. In Section \ref{1er}, we construct solutions in $H^1(\mathbb{R}^d)$ for any sign of $\lambda$, including both \eqref{NLS} in the subcritical case and \eqref{NLSlp} in the critical case, and we also provide the proof of Theorem \ref{08}. 
Section \ref{2em} deals with the construction of solutions in $W_1(\mathbb{R}^d)$ under the same assumptions, 
while Section \ref{3em} focuses on the construction of solutions in $\Sigma_\alpha$. 
The Appendix contains several technical lemmas along with their proofs.

\textbf{Notation:} We use the notation $A \lesssim_* B$ to denote the inequality $A \leq C(*) B$, and $\Omega \Subset \mathbb{R}^d$ means that $\Omega$ is an open set and its closure is compact.  For $f,g\in L^2(\mathbb{R}^d)$,  $$\Re(f,g):=\Re(\int f\overline{g})\quad \text{ and }\quad \Im(f,g):=\Im(\int f\overline{g}).$$
\section[Construction of the solution in H1(Rd)]{Construction of the solution in $ H^1(\mathbb{R}^d) $
}\label{1er}
Throughout this section until subsection \ref{l2}, we assume that $\lambda \in \mathbb{R}^*$, that \ref{A1}--\ref{A2} are satisfied, and that either \ref{B2} holds with $0 < \sigma < \frac{2}{d-2}$, or \ref{B1} holds with $0 < \sigma < \frac{2}{d}$, unless otherwise stated.
The construction of the solution in $ H^1(\mathbb{R}^d) $ relies on a compactness argument.
First, we regularize equation~\eqref{NLS}.

\begin{equation}\label{ppe}
      \left\{
    \begin{array}{ll}
         i\partial_t u_{\varepsilon} + \Delta u_{\varepsilon} + \lambda u_{\varepsilon} \log(|u_{\varepsilon}|^2 + \varepsilon) + u_\varepsilon g(|u_\varepsilon|^{2})= 0,  \\
       u_{\varepsilon}(0,x) = u_0,
    \end{array}
\right.
\tag{NLSlog$_\varepsilon$-g}
\end{equation}
for $0<\varepsilon<1$.
The following theorem can be established by means of Kato's method, as developed in \cite[Chapter 4]{1}.
\begin{theorem}\label{ppp}
Let $ u_0 \in H^1(\mathbb{R}^d) $.  
The Cauchy problem admits a unique maximal solution
$$
u_\varepsilon \in \mathcal{C}\bigl([0,T_{\max}^\varepsilon), H^1(\mathbb{R}^d)\bigr)
\cap 
\mathcal{C}^1\bigl([0,T_{\max}^\varepsilon), H^{-1}(\mathbb{R}^d)\bigr),
$$
together with the usual blow-up criterion and the following conservation laws.

\medskip
$\bullet$ \textit{Mass:}
$$
M(u) := \frac{1}{2}\int_{\mathbb{R}^d} |u|^2  \mathrm{d}x.
$$

\medskip
$\bullet$ \textit{Energy:}
$$
E_\varepsilon(u)
:= \frac{1}{2}\int_{\mathbb{R}^d} |\nabla u|^2  \mathrm{d}x
- \frac{\lambda}{2}\int_{\mathbb{R}^d} |u|^2 \log\bigl(|u|^2+\varepsilon\bigr)  \mathrm{d}x
+ \lambda \int_{\mathbb{R}^d} \eta(|u|)  \mathrm{d}x
- \int_{\mathbb{R}^d} G(|u|^2)  \mathrm{d}x,
$$
where
$$
\eta(s) := \int_0^s \frac{z^3}{z^2+\varepsilon}  \mathrm{d}z,
\qquad s \ge 0.
$$
\end{theorem}

\begin{remark}\label{mmm}
For fixed $0<\varepsilon<1$, if $ u \in H^1(\mathbb{R}^d) $, then:

\medskip
\noindent\textit{i)} $ |u|^2 \log\bigl(|u|^2+\varepsilon\bigr) \in L^1(\mathbb{R}^d) $.
Indeed, for any $ \delta > 0 $, one has
$$
|u|^2 \log(\varepsilon)
\le |u|^2 \log\bigl(|u|^2+\varepsilon\bigr)
\le |u|^2 \log\bigl(|u|^2+1\bigr)
\le \frac{2}{\delta} |u|^{2+\delta}.
$$
Hence,
$$
\bigl| |u|^2 \log\bigl(|u|^2+\varepsilon\bigr) \bigr|
\le \frac{2}{\delta} |u|^{2+\delta}
- \log(\varepsilon) |u|^2.
$$
Choosing $ \delta $ such that $ 2+\delta \in (2,2^*) $, the conclusion follows from the Sobolev embedding.

\noindent\textit{ii)}  We also have
$$
\int_0^{|u|} \frac{z^3}{z^2+\varepsilon}  \mathrm{d}z \in L^1(\mathbb{R}^d).
$$
Indeed, for $ z>0 $, we have
$$
\left| \frac{z^3}{z^2+\varepsilon} \right| \le z,
$$
and therefore
$$
\int_0^{|u|} \frac{z^3}{z^2+\varepsilon}  \mathrm{d}z
\le \frac{|u|^2}{2}.
$$
\end{remark}
\subsection{Uniform estimates}
We already know that the mass is conserved. It remains to obtain a uniform bound on the quantity
$ \norm{\nabla u_\varepsilon(t)}_{L^2} $.
By conservation of the energy, we have
$$
\frac{\mathrm{d}}{\mathrm{d}t}
\left(
\frac12 \|\nabla u_\varepsilon\|_{L^2}^2
- \frac{\lambda}{2} \int |u_\varepsilon|^2 \log\bigl(\varepsilon + |u_\varepsilon|^2\bigr)\mathrm{d}x
- \int G(|u_\varepsilon(t)|^2)\mathrm{d}x
\right)
=
-\frac{\lambda}{2} \int \frac{|u_\varepsilon|^2}{\varepsilon + |u_\varepsilon|^2}
 \partial_t\bigl(|u_\varepsilon|^2\bigr)\mathrm{d}x.
$$
In~\cite{13}, the authors first integrate in time and then use the sign condition $ \mu < 0 $ to eliminate the power-type term.
They also work in the space $ \Sigma_\alpha $ in order to estimate the logarithmic part.
Subsequently, they apply an integral version of Grönwall's inequality to estimate the quantity
$ \norm{\nabla u_\varepsilon(t)}_{L^2} $.\\
In our case, we estimate directly the time derivative of the logarithmic term.
This approach allows us to construct the solution in the space $ W_1(\mathbb{R}^d) $.

\begin{lemma}\label{jamila}
For all $t \in [0, T_{\max}^{\varepsilon})$, we have
    $$ \abs{\frac{\mathrm{d}}{\mathrm{d}t}\int \abs{u_\varepsilon(t)}^2\log(\abs{u_\varepsilon(t)}^2+\varepsilon)\mathrm{d}x}\leq 5 \norm{\nabla u_\varepsilon(t)}_{L^2}^2 \quad \text{ and } \quad \abs{\frac{\mathrm{d}}{\mathrm{d}t}\int \eta(\abs{u_\varepsilon(t)})\mathrm{d}x}\leq \frac{1}{2}\norm{\nabla u_\varepsilon(t)}_{L^2}^2.$$
\end{lemma}
\begin{proof}
  \begin{align*}
\abs{\frac{\mathrm{d}}{\mathrm{d}t}\int \abs{u_\varepsilon(t)}^2\log(\abs{u_\varepsilon(t)}^2+\varepsilon)\mathrm{d}x}
&\leq \abs{\int\log(\abs{u_\varepsilon(t)}^2+\varepsilon)\partial_t\abs{u_\varepsilon(t)}^2\mathrm{d}x}
+ \abs{\int\frac{\abs{u_\varepsilon(t)}^2}{\abs{u_\varepsilon(t)}^2+\varepsilon}\partial_t\abs{u_\varepsilon(t)}^2\mathrm{d}x} \\
&\leq 2\abs{\int\log(\abs{u_\varepsilon(t)}^2+\varepsilon)\Im(\overline{u_\varepsilon(t)}\Delta u_\varepsilon(t))\mathrm{d}x}\\&
\qquad+ 2\abs{\int\frac{\abs{u_\varepsilon(t)}^2}{\abs{u_\varepsilon(t)}^2+\varepsilon}\Im(\overline{u_\varepsilon(t)}\Delta u_\varepsilon(t))\mathrm{d}x} \\
&=2\abs{\Im\int \frac{2\overline{u_\varepsilon(t)}}{\abs{u_\varepsilon(t)}^2+\varepsilon}\Re\big( \overline{u_\varepsilon(t)}\nabla u_\varepsilon(t)\big)\cdot\nabla u_\varepsilon(t)\mathrm{d}x}
\\&\qquad+ 2\abs{\Im\int \frac{2\varepsilon\overline{u_\varepsilon(t)}}{(\abs{u_\varepsilon(t)}^2+\varepsilon)^2}\Re\big( \overline{u_\varepsilon(t)}\nabla u_\varepsilon(t)\big)\cdot\nabla u_\varepsilon(t)\mathrm{d}x} \\
&\leq  4\norm{\nabla u_\varepsilon(t)}_{L^2}^2 + \norm{\nabla u_\varepsilon(t)}_{L^2}^2.
\end{align*}
In the last inequality, we used $\frac{a^2 \varepsilon}{(a^2 + \varepsilon)^2} \leq \frac{1}{4}$.
\\
For the second inequality, we have
$$
\eta(s)=\int_0^s\frac{z^3}{z^2+\varepsilon}\mathrm{d}z=\frac{1}{2}\int_0^{s^2}\frac{z}{z+\varepsilon}\mathrm{d}z.
$$
Hence, with a similar computation as the previous case,
\begin{equation*}
    \abs{\frac{\mathrm{d}}{\mathrm{d}t}\int\eta(\abs{u_\varepsilon(t)})\mathrm{d}x}
\leq \abs{\int \frac{\abs{u_\varepsilon(t)}^2}{\abs{u_\varepsilon(t)}^2+\varepsilon}\Im(\overline{u_\varepsilon(t)}\Delta u_\varepsilon(t))\mathrm{d}x}
\leq \frac{1}{2} \norm{\nabla u_\varepsilon(t)}_{L^2}^2.\qedhere
\end{equation*}
\end{proof}
\begin{lemma}\label{kpo}
For all $t \in [0, T_{\max}^{\varepsilon})$, we have
 $$\frac{\mathrm{d}}{\mathrm{d}t}\left(\frac{1}{2}\norm{\nabla u_\varepsilon(t)}^2_2- \int G(|u_\varepsilon(t)|^2)\mathrm{d}x\right)\leq 3\abs{\lambda} \norm{\nabla u_\varepsilon(t)}_{L^2}^2.$$
\end{lemma}
\begin{proof}
 We know that the energy is conserved along the flow, so
\begin{align*}
\frac{\mathrm{d}}{\mathrm{d}t}\left(\frac{1}{2}\norm{\nabla u_\varepsilon(t)}^2_2- \int G(|u_\varepsilon(t)|^2)\mathrm{d}x\right)
&=\frac{\mathrm{d}}{\mathrm{d}t}\left(\frac{\lambda}{2}\int \abs{u_\varepsilon(t)}^2\log(\abs{u_\varepsilon(t)}^2+\varepsilon)\mathrm{d}x-\lambda \int \eta(\abs{u_\varepsilon(t)})\mathrm{d}x\right).
\end{align*}
By Lemma~\ref{jamila}, we deduce that 
\begin{equation*}
    \frac{\mathrm{d}}{\mathrm{d}t}\left(\frac{1}{2}\norm{\nabla u_\varepsilon(t)}^2_2- \int G(|u_\varepsilon(t)|^2)\mathrm{d}x\right)\leq 3\abs{\lambda}\norm{\nabla u_\varepsilon(t)}_{L^2}^2. \qedhere
\end{equation*}
\end{proof}
\begin{lemma}\label{1}
    For all $t \in [0,T_{\max}^{\varepsilon})$, we have
$$
\norm{\nabla u_\varepsilon(t)}_{L^2} \leq C e^{C t} \norm{\nabla u_0}_{L^2},
$$
for some constant $C := C(\norm{u_0}_{H^1(\mathbb{R}^d)}) > 0$.
\end{lemma}
\begin{proof}
    From Lemma \ref{kpo}, we have
    \begin{equation}\label{df}
        \frac{\mathrm{d}}{\mathrm{d}t}\left(\frac{1}{2}\norm{\nabla u_\varepsilon(t)}^2_2 - \int G(|u_\varepsilon(t)|^2)\mathrm{d}x \right)
\leq 3 |\lambda|  \norm{\nabla u_\varepsilon(t)}_{L^2}^2.
    \end{equation}
Integrating in time, we obtain
$$
\norm{\nabla u_{\varepsilon}(t)}_{L^2}^2 \lesssim \int_0^t \norm{\nabla u_\varepsilon(z)}_{L^2}^2  \mathrm{d}z +\int G(|u_\varepsilon(t)|^2)\mathrm{d}x+C(\norm{ u_0}_{H^1}).
$$
We distinguish between the cases corresponding to conditions \ref{B1} and \ref{B2}.

Suppose that condition~\ref{B1} holds, then
$$
\norm{\nabla u_{\varepsilon}(t)}_{L^2}^2 \lesssim \int_0^t \norm{\nabla u_\varepsilon(z)}_{L^2}^2  \mathrm{d}z +C(\norm{ u_0}_{H^1}).
$$
Grönwall's lemma allows us to conclude.

Suppose that condition~\ref{B2} holds,
by Corollary \ref{power10} and Gagliardo-Nirenberg inequality, we deduce that
\begin{align*}
\norm{\nabla u_{\varepsilon}(t)}_{L^2}^2
&\lesssim \int_0^t \norm{\nabla u_\varepsilon(z)}_{L^2}^2  \mathrm{d}z + \norm{u_{\varepsilon}(t)}_{L^2}^{2\sigma+2-d\sigma} \norm{\nabla u_\varepsilon(t)}_{L^2}^{d\sigma}+C(\norm{ u_0}_{H^1}) \\
&\lesssim \int_0^t \norm{\nabla u_\varepsilon(z)}_{L^2}^2  \mathrm{d}z + \norm{u_0}_{L^2}^{2\sigma+2-d\sigma} \norm{\nabla u_\varepsilon(t)}_{L^2}^{d\sigma}+C(\norm{ u_0}_{H^1}).
\end{align*}
Applying Lemma~\ref{fait12}, we deduce that
$$
\norm{\nabla u_{\varepsilon}(t)}_{L^2}^2 \lesssim \int_0^t \norm{\nabla u_\varepsilon(z)}_{L^2}^2  \mathrm{d}z + C(\norm{ u_0}_{H^1}).
$$
Finally, Grönwall's lemma allows us to conclude.
\end{proof}
\begin{remark}
   Using this lemma and the conservation of mass, we easily show that the constructed solution is global in time, that is, $T_{\max}^{\varepsilon} = +\infty$.
\end{remark}
For $T>0$, set
\begin{equation}\label{nt}
N_T:=\sup_{\substack{0 < \varepsilon < 1 \\ t \in [0,T]}}
\norm{u_\varepsilon(t)}_{H^1(\mathbb{R}^d)}< \infty.
\end{equation}

\subsection{Convergence of the sequence}
In this subsection, we will rely on a compactness argument of Aubin-Lions (see Lemma \ref{aubin}) to show that, up to extraction, the sequence $(u_\varepsilon)_\varepsilon$ converges in $\mathcal{C}_{\rm loc}(\mathbb{R}_+, L^2_{\text{loc}}(\mathbb{R}^d))$.
\begin{lemma}\label{limu}
   The sequence (up to extraction) $(u_\varepsilon)_\varepsilon$ converges in $\mathcal{C}_{\rm loc }(\mathbb{R}_+, L^2_{\text{loc}}(\mathbb{R}^d))$.
\end{lemma}
\begin{proof}
   Let $T>0$. We denote by $B_r$ the open ball of radius $r>0$ in $\mathbb{R}^d$. By Sobolev embeddings, we have $L^{\frac{2\sigma+2}{2\sigma+1}}(B_r) \hookrightarrow H^{-1}(B_r)$ and $L^2(B_r) \hookrightarrow H^{-1}(B_r)$, and assumption \ref{A2}
\begin{align*}
\|\partial_t u_\varepsilon\|_{H^{-1}(B_r)} &\lesssim \norm{\nabla u_\varepsilon}_{L^2(B_r)} 
+ |\lambda|\|u_\varepsilon \log(|u_\varepsilon|^2+\varepsilon)\|_{L^2(B_r)} 
+ \|u_\varepsilon g(|u_\varepsilon|^{2})\|_{L^{\frac{2\sigma+2}{2\sigma+1}}(B_r)} \\
&\lesssim N_T +\|u_\varepsilon \log(|u_\varepsilon|^2+\varepsilon)\|_{L^2(B_r)} 
+  \norm{u_\varepsilon}_{L^{2\sigma+2}(B_r)}^{2\sigma+1}.
\end{align*}

\begin{itemize}
\item For the logarithmic term: For all $t \in [0,T]$ and any $\delta>0$, we have
$$
\norm{u_\varepsilon(t) \log(|u_\varepsilon(t)|^2 + \varepsilon)}_{L^2(B_r)}^2
\le C(\delta) \left( \int_{B_r} |u_\varepsilon(t)|^2\mathrm{d}x + \int_{B_r} |u_\varepsilon(t)|^{2 - 2\delta}\mathrm{d}x + \int_{B_r} |u_\varepsilon(t)|^{2 + 2\delta}\mathrm{d}x \right).
$$
By the inclusion $L^2(B_r) \subset L^{2-\delta}(B_r)$, the Sobolev embedding with $\delta$ small enough, and using (\ref{nt}), we deduce
$$
\norm{u_\varepsilon(t) \log(|u_\varepsilon(t)|^2 + \varepsilon)}_{L^2(B_r)}^2 \le C(N_T).
$$

\item For the last term: it suffices to use the Sobolev embedding $H^1(B_r) \hookrightarrow L^{2\sigma+2}(B_r)$.
\end{itemize}
In conclusion:
$$
\norm{\partial_t u_\varepsilon}_{H^{-1}(B_r)} \le C(N_T).
$$
Moreover, we know that
$$
H^1(B_r) \overset{\text{compact}}{\hookrightarrow} L^2(B_r) \overset{\text{continuous}}{\hookrightarrow} H^{-1}(B_r).
$$
By Aubin-Lions compactness Theorem \ref{aubin} and diagonal extraction, we get the conclusion.
\end{proof}
\begin{remark}\label{xv}
 Up to extraction, we have
$$
u_\varepsilon(t) \rightharpoonup u(t) \quad \text{in } L^2(\mathbb{R}^d), \quad \text{for all } t \in \mathbb{R}_+.
$$
Thus,
$$
\norm{u(t)}_{L^2} \le \liminf_{\varepsilon \to 0} \norm{u_\varepsilon(t)}_{L^2} = \norm{u_0}_{L^2}.
$$
In particular $u\in L^{\infty}(\mathbb{R}_+,L^2(\mathbb{R}^d))$. The conservation of the mass will be proved in the following subsection.
\end{remark}

 \begin{lemma}\label{1.2}
  The function $u$ defined in Lemma~\ref{limu} as the limit of $u_\varepsilon$, satisfies $u \in L^{\infty}_{\rm loc}(\mathbb{R}_+, H^1(\mathbb{R}^d))$, and for all $t \in \mathbb{R}_+$ (up to extraction): 
$$
u_\varepsilon(t) \rightharpoonup u(t) \quad \text{in } H^1(\mathbb{R}^d)\quad  \text{ and }\quad u_\varepsilon(t) \log(|u_\varepsilon(t)|^2 + \varepsilon) \to u(t) \log|u(t)|^2 \quad \text{in } L^2_{\text{loc}}(\mathbb{R}^d).
$$
Moreover, for all $\Omega \Subset \mathbb{R}^d$, we have
$$
u_\varepsilon(t) g(|u_\varepsilon(t)|^{2}) \to u(t) g(|u(t)|^{2}) \quad \text{in } H^{-1}(\Omega).
$$
\end{lemma}
\begin{proof}
 We begin with the first convergence.  
By Remark \ref{xv}, we have
$$
u_\varepsilon(t) \rightharpoonup u(t) \quad \text{in } L^2(\mathbb{R}^d), \quad \text{for all } t \in \mathbb{R}_+.
$$
On the other hand, the sequence $u_\varepsilon(t)$ is uniformly bounded with respect to $\varepsilon$ in $H^{1}(\mathbb{R}^d)$, which yields the desired first convergence.

 We then turn to the second convergence. Let $\Omega \Subset \mathbb{R}^d$. By Lemma~\ref{cv2}, for $0 < \delta_{1,2} < \frac{1}{4}$, we have
$$
 \begin{aligned}
&\norm{ u_\varepsilon(t) \log(|u_\varepsilon(t)|^2+\varepsilon) - u(t) \log|u(t)|^2}_{L^2(\Omega)}^2
\\&~~~~~~~~\lesssim_{\delta_{1,2},\Omega} \varepsilon^2 + \norm{u(t) - u_\varepsilon(t)}_{L^2(\Omega)}^2+ \norm{|u(t) - u_\varepsilon(t)|^{1/2} \big( |u_\varepsilon(t)|^{1/2 - \delta_1}+|u(t)|^{1/2 - \delta_1})}_{L^2(\Omega)}^2\\&\qquad ~~~~~~~~~+\norm{|u(t) - u_\varepsilon(t)|^{1/2} \big( |u_\varepsilon(t)|^{1/2 + \delta_2}+|u(t)|^{1/2 + \delta_2})}_{L^2(\Omega)}^2.
\end{aligned}
$$

$\circ$ The first two terms converge to $0$ as $\varepsilon\to 0$. 

$\circ$ \textit{Third term:}
\begin{align*}
&\norm{\Big(|u_\varepsilon(t)|^{1/2 - \delta_1} +|u(t)|^{1/2 - \delta_1}\Big)|u_\varepsilon(t) - u(t)|^{1/2}}_{L^2(\Omega)}^2
\\&\qquad\lesssim \Big(\norm{|u_\varepsilon(t)|^{1 - 2\delta_1}}_{L^2(\Omega)}+\norm{|u(t)|^{1 - 2\delta_1}}_{L^2(\Omega)}\Big) \norm{u_\varepsilon(t) - u(t)}_{L^2(\Omega)} \\
&\qquad= \Big(\norm{u_\varepsilon(t)}_{L^{2 - 4\delta_1}(\Omega)}^{1 - 2\delta_1}+\norm{u(t)}_{L^{2 - 4\delta_1}(\Omega)}^{1 - 2\delta_1}\Big) \norm{u_\varepsilon(t) - u(t)}_{L^2(\Omega)}.
\end{align*}
Since $L^2(\Omega) \subset L^{2 - 4\delta_1}(\Omega)$ and by $\norm{u_\varepsilon(t)}_{L^2} \le \norm{u_0}_{L^2}$, we have
$$
\norm{\Big(|u_\varepsilon(t)|^{1/2 - \delta_1} +|u(t)|^{1/2 - \delta_1}\Big)|u_\varepsilon(t) - u(t)|^{1/2}}_{L^2(\Omega)}^2 \lesssim C(\norm{u_0}_{L^2}) \norm{u_\varepsilon(t) - u(t)}_{L^2(\Omega)}.
$$
Hence,
$$
\norm{|u_\varepsilon(t)|^{1/2 - \delta_1} |u_\varepsilon(t) - u(t)|^{1/2}}_{L^2(\Omega)}^2 \underset{\varepsilon \to 0}{\longrightarrow} 0.
$$

$\circ$ \textit{Fourth term:}
\begin{align*}
&\norm{\Big(|u_\varepsilon(t)|^{1/2 + \delta_2}+|u(t)|^{1/2 + \delta_2}\Big) |u_\varepsilon(t) - u(t)|^{1/2}}_{L^2(\Omega)}^2
\\&\qquad\lesssim \Big(\norm{|u_\varepsilon(t)|^{1 + 2\delta_2}}_{L^2(\Omega)}+\norm{|u_\varepsilon(t)|^{1 + 2\delta_2}}_{L^2(\Omega)}\Big) \norm{u_\varepsilon(t) - u(t)}_{L^2(\Omega)} \\
&\qquad= \Big(\norm{u_\varepsilon(t)}_{L^{2 + 4\delta_2}(\Omega)}^{1 + 2\delta_2}+\norm{u_\varepsilon(t)}_{L^{2 + 4\delta_2}(\Omega)}^{1 + 2\delta_2}\Big) \norm{u_\varepsilon(t) - u(t)}_{L^2(\Omega)}.
\end{align*}
We choose $\delta_2$ such that $2 + 4\delta_2 < 2^*$. By the Sobolev embedding, we conclude
$$
\norm{\Big(|u_\varepsilon(t)|^{1/2 + \delta_2}+|u(t)|^{1/2 + \delta_2}\Big) |u_\varepsilon(t) - u(t)|^{1/2}}_{L^2(\Omega)}^2 \lesssim C(\norm{u_0}_{H^1(\mathbb{R}^d)}) \norm{u_\varepsilon(t) - u(t)}_{L^2(\Omega)}.
$$
But $u_\varepsilon(t) \to u(t)$ in $L^2(\Omega)$ for all $t \in \mathbb{R}_+$. Combining this with (\ref{nt}), we deduce the convergence stated in the lemma.

Finally, we prove the last convergence. Let $\Omega \Subset \mathbb{R}^d$. Using Corollary \ref{power10}, Hölder's inequality, and the Sobolev embedding $L^{\frac{2\sigma+2}{2\sigma+1}}(\Omega) \hookrightarrow H^{-1}(\Omega)$ and $L^2(\Omega)\hookrightarrow H^{-1}(\Omega)$ we deduce that
\begin{align*}
    &\norm{u_\varepsilon(t) g(|u_\varepsilon(t)|^{2}) - u(t) g(|u(t)|^{2})}_{H^{-1}(\Omega)}
\\&~~~~~~~~~~~~~~~~\lesssim \left(\norm{u_\varepsilon(t)}_{L^{2\sigma+2}(\Omega)}^{2\sigma} + \norm{u(t)}_{L^{2\sigma+2}(\Omega)}^{2\sigma} \right) \norm{u_\varepsilon(t) - u(t)}_{L^{2\sigma+2}(\Omega)}+\norm{u_\varepsilon(t)-u(t)}_{L^2(\Omega)}.
\end{align*}

By the embedding $H^1(\Omega) \hookrightarrow L^{2\sigma+2}(\Omega)$ and $L^{2\sigma+2}(\Omega)\subset L^2(\Omega)$, we obtain
$$
\norm{u_\varepsilon(t) g(|u_\varepsilon(t)|^{2}) - u(t) g(|u(t)|^{2})}_{H^{-1}(\Omega)} \lesssim_{T,\Omega} \norm{u_\varepsilon(t) - u(t)}_{L^{2\sigma+2}(\Omega)}.
$$
Moreover, by Lemma~\ref{1.2}, we have weak convergence $u_\varepsilon(t) \rightharpoonup u(t)$ in $H^1(\Omega)$, so by the Rellich-Kondrachov compactness theorem, we have
\begin{equation*}
    u_\varepsilon(t) \to u(t) \quad \text{strongly in } L^{2\sigma+2}(\Omega).
\qedhere
\end{equation*}
\end{proof}
\begin{lemma}\label{ghj}
    The function $u$ defined in Lemma~\ref{limu} as the limit of $u_\varepsilon$ satisfies, for all $\Omega \Subset \mathbb{R}^d$,
$$
u \in L^\infty_{\text{loc}}(\mathbb{R}_+, H^1(\mathbb{R}^d)) \cap W_{\text{loc}}^{1,\infty}(\mathbb{R}_+, H^{-1}(\Omega)),
$$
and moreover,
\begin{equation*}\label{equation}
i u_t + \Delta u + \lambda u \log|u|^2 +  u g(|u|^{2}) = 0 \quad \text{in } H^{-1}(\Omega), \text{ for all } t \in \mathbb{R}_+.
\end{equation*}
\end{lemma}
\begin{proof}
    Let $\phi \in \mathcal{C}_c^1([0,+\infty))$ and $\psi \in \mathcal{C}_c^{\infty}(\mathbb{R}^d)$, then
\begin{align*}
    \int_0^{\infty} (i u_\varepsilon, \psi)_2  \phi'(t)  dt 
    = \int_0^{\infty} \Big[ -(\nabla u_\varepsilon, \nabla \psi)_2 
    + \lambda (f_{1\varepsilon}(u_\varepsilon), \psi)_2 
    +  \langle f_{2\varepsilon}(u_\varepsilon), \psi\rangle_{H^{-1}(\mathbb{R}^d),H^1(\mathbb{R}^d)} \Big] \phi(t)  dt,
\end{align*}
where $f_{1\varepsilon}(u_\varepsilon) = u_\varepsilon \log(|u_\varepsilon|^2 + \varepsilon)$ and $f_{2\varepsilon}(u_\varepsilon) = u_\varepsilon g(|u_\varepsilon|^{2})$.

$\bullet$ \textit{First term:} By Cauchy-Schwarz, we have
$$
\left| (\nabla u_\varepsilon(t), \nabla \psi)_2 \right| \leq N_T \| \nabla \psi \|_{L^2}.
$$

$\bullet$ \textit{Second term:} From the previous part, we have
$$
\norm{u_\varepsilon(t) \log(|u_\varepsilon(t)|^2 + \varepsilon)}_{L^2(\Omega)}^2 \leq C(N_T),
$$
where $\Omega = \text{supp}(\psi)$. Hence,
$$
\left| (f_{1\varepsilon}(u_\varepsilon(t)), \psi)_2 \right| \leq C(N_T) \| \psi \|_{L^2}.
$$

$\bullet$ \textit{Third term:} By Hölder's inequality with $(p,p')=(r',r)$,
$$
\left| (f_{2\varepsilon}(u_\varepsilon(t)), \psi)_2 \right| 
\leq \| u_\varepsilon \|_{L^{2\sigma+2}}^{2\sigma+1} \| \psi \|_{L^{2\sigma+2}}.
$$
By the Sobolev embedding $H^1(\mathbb{R}^d) \hookrightarrow L^r(\mathbb{R}^d)$, we obtain
$$
\left| (f_{2\varepsilon}(u_\varepsilon(t)), \psi)_2 \right| 
\lesssim N_T^{2\sigma+1} \| \psi \|_{L^{2\sigma+2}}.
$$
Finally, combining the dominated convergence theorem with Lemma~\ref{1.2}, we deduce that
$$
\int_0^{\infty} (i u, \psi)_2 \phi'(t)  dt 
= \int_0^{\infty} \Big[ -(\nabla u, \nabla \psi)_2 
+ \lambda (f_1(u), \psi)_2 
+  \langle f_2(u), \psi \rangle_{H^{-1}(\mathbb{R}^d),H^1(\mathbb{R}^d)} \Big] \phi(t)  dt.
$$
From the above, we have for all $\Omega \Subset \mathbb{R}^d$
$$
i u_t + \Delta u + \lambda u \log(|u|^2) +  u g(|u|^{2}) = 0 \quad \text{in } H^{-1}(\Omega), \text{ for all } t \in \mathbb{R}_+.
$$
To deduce that $u \in W_{\text{loc}}^{1,\infty}(\mathbb{R}_+, H^{-1}(\Omega))$, it suffices to notice that the map 
$$
L: u \mapsto \Delta u + \lambda u \log|u|^2 +  u g(|u|^{2}),
$$ 
is bounded from $H^1(\mathbb{R}^d)$ into $H^{-1}(\Omega)$.
\end{proof}
\subsection{Mass conservation and regularity of the solution}
In \cite{3}, the conservation of mass is obtained thanks to the uniqueness of the solution and the regularity $u \in \mathcal{C}_w(\mathbb{R}_+, H^1)$, which we have not yet established. In our case, we use a different approach by passing through a cut-off function.
\\
Let $\chi \in \mathcal{C}_c^{1}(\mathbb{R}^d)$ be a cut-off function satisfying
$$
\chi(x) =
\begin{cases}
1 & \text{if } |x| \le \frac{1}{2}, \\
0 & \text{if } |x| \ge 1,
\end{cases}
\quad \text{with } 0 \le \chi(x) \le 1 \text{ for all } x \in \mathbb{R}^d.
$$
For $R > 0$, we set $\chi_R(x) := \chi\bigl(\frac{x}{R}\bigr)$.
\begin{lemma}\label{23}
   The function $u$, solution of equation \eqref{NLS} defined in Lemma \ref{ghj}, satisfies
$$
\norm{u(t)}_{L^2} = \norm{u_0}_{L^2},
$$
for all $t \in \mathbb{R}_+$.  
In particular, up to a subsequence,
$$
u_{\varepsilon}(t) \to u(t) \quad \text{in } L^2(\mathbb{R}^d), \quad \text{for all } t \in \mathbb{R}_+.
$$
\end{lemma}
\begin{proof}
Let $t \in \mathbb{R}_+$. Then
\begin{align*}
   \abs{ \frac{\mathrm{d}}{\mathrm{d}t} \norm{\chi_R u(t)}^2_2 }
   = 2 \abs{ \Im \bigl( \nabla (\chi_R)^2 \nabla u(t), u(t) \bigr) } 
   \leq \frac{C}{R} \norm{\nabla u(t)}_{L^2} \norm{u(t)}_{L^2} 
   \leq \frac{C}{R} \norm{u_0}_{L^2} N_T.
\end{align*}
Integrating both sides over $[0,t]$, we deduce:
$$
\abs{ \norm{\chi_R u(t)}_{L^2}^2 - \norm{\chi_R u_0}_{L^2}^2 } \leq \frac{C}{R} T \norm{u_0}_{L^2} N_T.
$$
By the dominated convergence theorem, we then conclude
$$
\norm{u(t)}_{L^2} = \norm{u_0}_{L^2}.
$$
Using the weak convergence obtained in Lemma \ref{1.2}, we have that $u_\varepsilon(t)$ converges weakly to $u(t)$ in $L^2(\mathbb{R}^d)$. On the other hand, since $\norm{u_\varepsilon(t)}_{L^2} = \norm{u_0}_{L^2} = \norm{u(t)}_{L^2}$, we obtain the strong convergence of the sequence $(u_\varepsilon(t))_\varepsilon$ in $L^2(\mathbb{R}^d)$.
\end{proof}

\begin{corollary}\label{rq2}
The function $u$, solution of equation \eqref{NLS} defined in Lemma \ref{ghj}, satisfies, up to a subsequence
$$
u_\varepsilon(t) \to u(t) \quad \text{in } L^{2\sigma+2}(\mathbb{R}^d), \quad \text{for all } t \in \mathbb{R}_+.
$$
In particular, $$
G(|u_\varepsilon(t)|^2) \to G(|u(t)|^2) \quad \text{in } L^1(\mathbb{R}^d), \quad \text{for all } t \in \mathbb{R}_+.
$$
\end{corollary}
\begin{proof}
For the first convergence, we use the Lemma \ref{23}, and the fact that $u_\varepsilon(t)$ is uniformly bounded in $H^1(\mathbb{R}^d)$ together with the Gagliardo–Nirenberg inequality. For the second convergence we obtain the desired conclusion by using Corollary \ref{power10} and the first convergence. 
\end{proof}
\begin{lemma}\label{continue}
The function $u$, solution of equation \eqref{NLS} defined in Lemma \ref{ghj}, satisfies
$$
u \in \mathcal{C}\bigl(\mathbb{R}_+, H^1(\mathbb{R}^d)\bigr).
$$
\end{lemma}
\begin{proof}
  First we note that $u \in \mathcal{C}_w(\mathbb{R}_+, H^1(\mathbb{R}^d))$. Indeed, this easily follows from Lemma \ref{1.2} and $u \in \mathcal{C}(\mathbb{R}_+,L^2_{\rm loc}(\mathbb{R}^d))$, therefore, by mass conservation, we conclude that $u \in \mathcal{C}(\mathbb{R}_+, L^2(\mathbb{R}^d))$. We combine the Gagliardo-Nirenberg inequality \ref{GN} with $u \in L^\infty_{\rm loc}(\mathbb{R}_+, H^1(\mathbb{R}^d)) \cap \mathcal{C}(\mathbb{R}_+, L^{2}(\mathbb{R}^d))$, we get $u \in \mathcal{C}(\mathbb{R}_+, L^{2\sigma+2}(\mathbb{R}^d))$, and $G(|u|^2)\in \mathcal{C}(\mathbb{R}_+,L^1(\mathbb{R}^d))$.

Next we show $t\mapsto\norm{\nabla u(t)}_{L^2}\in \mathcal{C}(\mathbb{R}_+)$. By uniqueness it is enough to show the continuity at $t=0.$ \\
By Lemma \ref{kpo}, we have $$\frac{\mathrm{d}}{\mathrm{d}t}\left(\frac{1}{2}\norm{\nabla u_\varepsilon(t)}^2_2- \int G(|u_\varepsilon(t)|^2)\mathrm{d}x\right)\leq 3\abs{\lambda} \norm{\nabla u_\varepsilon(t)}_{L^2}^2.$$ Hence, $$\frac{1}{2}\norm{\nabla u_\varepsilon(t)}^2_2- \int G(|u_\varepsilon(t)|^2)\mathrm{d}x\leq C(e^{Ct}-1)+\frac{1}{2}\norm{\nabla u_0}^2_2-\int G(|u_0|^2)\mathrm{d}x.$$
By Lemma \ref{1.2} and Corollary \ref{rq2}, we get 
$$\frac{1}{2}\norm{\nabla u(t)}^2_2- \int G(|u(t)|^2)\mathrm{d}x\leq C(e^{Ct}-1)+\frac{1}{2}\norm{\nabla u_0}^2_2- \int G(|u_0|^2)\mathrm{d}x.$$
Hence,  $$\limsup_{t\to 0}\norm{\nabla u(t)}^2_2\leq \norm{\nabla u_0}^2_2.$$
On the other hand $\nabla u\in \mathcal{C}_w(\mathbb{R}_+,L^2(\mathbb{R}^d))$. Therefore, combining with the previous calculation, we deduce that $t\mapsto\norm{\nabla u(t)}_{L^2}\in \mathcal{C}(\mathbb{R}_+)$.
\end{proof}
\subsection{Uniqueness and the continuity property of the flow}
In this subsection, we prove Theorem~\ref{08}.  
The key idea is to use the proof scheme introduced by Hayashi in \cite{14}, based on the use of the Duhamel formula. We write the Duhamel term as
$$
\Phi(h)(t) := \int_0^t \mathcal{S}(t-\tau) h(\tau)  d\tau, 
\qquad \mathcal{S}(t) = e^{i t \Delta}.
$$
We say that a pair $(q,r) \in [2,\infty]^2$ is \emph{admissible} if $(q,r,d) \neq (2,\infty,2)$ and
$
\frac{2}{q} = d \left( \frac{1}{2} - \frac{1}{r} \right).
$ 
\begin{lemma}\cite[Theorem 2.3.2]{1}
    Let $(q,r)$ and $(\gamma,\rho)$ be admissible pairs, and let $I$ be an interval. 
If $h \in L^{\gamma'}(I, L^{\rho'}(\mathbb{R}^d))$, then the Duhamel operator $\Phi(h)$ satisfies
$$
\|\Phi(h)\|_{L^q(I,L^r(\mathbb{R}^d))} \lesssim \|h\|_{L^{\gamma'}(I,L^{\rho'}(\mathbb{R}^d))}.
$$
\end{lemma}

We consider a function $\zeta \in \mathcal{C}^1_c(\mathbb{C}, \mathbb{R})$ satisfying
$$
\zeta(z) = 
\begin{cases}
1 & \text{if } |z| \le \frac{1}{2}, \\
0 & \text{if } |z| \ge 1,
\end{cases}
\quad \text{and} \quad 0 \le \zeta(z) \le 1 \text{ for all } z \in \mathbb{C}.
$$
We define
$$
g_1(z) = \zeta(z) z \log|z|^2, \quad \text{and} \quad g_2(z) = (1-\zeta(z)) z \log|z|^2.
$$

\begin{lemma}\label{uniq}
    Let $\alpha\in (0,1)$, $T>0$ and $I_t = [0,t] \subset [0,T]$.  
Let $u$ and $v$ be two solutions of \eqref{NLS} in $H^1(\mathbb{R}^d)$ with initial data $u_0$ and $v_0$, respectively.  
Set $w := u - v$. Then
    $$
\| \chi_R w \|_{X_t} \le \| \chi_R w_0 \|_{L^2} + C \frac{tM}{R} 
+ C (t^{2\sigma/\theta} M^{2\sigma}+t) \| \chi_Rw\|_{X_t} 
+ \frac{4C}{1-\alpha} |B_R|^{\frac{1-\alpha}{2}} \int_0^t \| \chi_R w \|_{L^2}^\alpha,
$$
where $$
X_t = L^\infty(I_t, L^2(\mathbb{R}^d)) \cap L^q(I_t, L^r(\mathbb{R}^d)), \quad
\norm{y}_{X_t} = \norm{y}_{L^\infty_t L^2_x} + \norm{y}_{L^q_t L^r_x},
$$
and
$$
M := \max\bigl( \|v\|_{L^\infty_T H^1(\mathbb{R}^d)}, \norm{u}_{L^\infty_T H^1(\mathbb{R}^d)} \bigr).
$$
\end{lemma}
\begin{proof}
        From the equation we have
$$
i \partial_t (\chi_R u) + \chi_R \Delta u + \lambda \chi_R u \log|u|^2 +  \chi_R u g(|u|^{2}) = 0 \quad \text{in } H^{-1}(\mathbb{R}^d).
$$
Now, using
$$
\Delta (\chi_R u) = u \Delta \chi_R + 2 \nabla u \cdot \nabla \chi_R + \chi_R \Delta u \quad \text{in } H^{-1}(\mathbb{R}^d),
$$
we obtain
$$
i \partial_t (\chi_R u) + \Delta (\chi_R u) - u\Delta \chi_R - 2 \nabla u\cdot \nabla \chi_R
+ \lambda \chi_R u \log|u|^2 +  \chi_R u g(|u|^{2}) = 0 \quad \text{in } H^{-1}(\mathbb{R}^d).
$$
By Duhamel's formula, we get
\begin{align*}
\chi_R u(t) &= \mathcal{S}(t) \chi_R u_0 + i \lambda \int_0^t \mathcal{S}(t-z) \chi_R u \log|u|^2  dz
+ i  \int_0^t \mathcal{S}(t-z) \chi_R ug(|u|^{2})  dz \\
&\quad - i \int_0^t \mathcal{S}(t-z) \bigl( u \Delta \chi_R + 2 \nabla u \cdot \nabla \chi_R \bigr)  dz.
\end{align*}
Then
\begin{align*}
        \chi_Rw(t)&=\mathcal{S}(t) \chi_R w_0 +i\lambda\int_0^t\mathcal{S}(t-z)\chi_R(u\log\abs{u}^2-v\log\abs{v}^2)\mathrm{d}z+i\int_0^t\mathcal{S}(t-z)\chi_R(ug(\abs{u}^{2})-vg(\abs{v}^{2}))\mathrm{d}z\\
        &~~~-i\int_0^t\mathcal{S}(t-z)(w\Delta \chi_R+2\nabla w\nabla \chi_R)\mathrm{d}z\\
        &=\mathcal{S}(t) \chi_R w_0 +i\lambda\int_0^t\mathcal{S}(t-z)\chi_R(g_1(u)-g_1(v))\mathrm{d}z+i\lambda\int_0^t\mathcal{S}(t-z)\chi_R(g_2(u)-g_2(v))\mathrm{d}z\\
        &~~~+i\int_0^t\mathcal{S}(t-z)\chi_R(ug(\abs{u}^{2})-vg(\abs{v}^{2}))\mathrm{d}z-i\int_0^t\mathcal{S}(t-z)(w\Delta \chi_R+2\nabla w\nabla \chi_R)\mathrm{d}z\\
        &=\mathcal{S}(t) \chi_R w_0 +e_1(t)+e_2(t)+e_3(t)+e_4(t).
    \end{align*}
where $e_1, \dots, e_4$ correspond to the four integral terms above.

$\circ$ \textit{Estimate of $e_4(t)$:} By Strichartz estimates,
 $$\norm{e_4}_{X_t}\lesssim \int_0^t\norm{w(z)\Delta \chi_R+2\nabla w(z)\nabla \chi_R}_{L^2}\mathrm{d}z\lesssim \frac{1}{R^2}\int_0^t \norm{w(z)}_{L^2}\mathrm{d}z+\frac{1}{R}\int_0^t \norm{\nabla w(z)}_{L^2}\mathrm{d}z\lesssim \frac{tM}{R}.$$

$\circ$ \textit{Estimate of $e_3(t)$:} Define $\theta:=\frac{d\sigma^2}{2\sigma+2-d\sigma}q$. Using H\"older and Sobolev embedding, and Strichartz estimates,
 \begin{align*}
        \norm{e_3}_{X_t}&\lesssim \norm{\chi_R(u|u|^{2\sigma}-v|v|^{2\sigma})}_{L^{q'}_tL^{r'}}+t\norm{\chi_R(u-v)}_{L^{\infty}_tL^{2}}
        \\&\lesssim t^{2\sigma/\theta}(\norm{u}_{L^{\infty}H^1(\mathbb{R}^d)}^{2\sigma}+\norm{v}_{L^{\infty}H^1(\mathbb{R}^d)}^{2\sigma})\norm{\chi_R(u-v)}_{L^q_tL^r}+t\norm{\chi_R(u-v)}_{L^{\infty}_tL^{2}}\\
        &\lesssim t^{2\sigma/\theta} M^{2\sigma}\norm{\chi_R(u-v)}_{L^q_tL^r}+t\norm{\chi_R(u-v)}_{L^{\infty}_tL^{2}}\\
        &\lesssim (t^{2\sigma/\theta} M^{2\sigma}+t)\norm{\chi_R(u-v)}_{X_t}.
    \end{align*}

$\circ$ \textit{Estimate of $e_2(t)$:} Applying Lemma~\ref{ii} with $\alpha = 2\sigma$, we similarly get
$$
\| e_2 \|_{X_t} \lesssim \norm{\chi_R(g_2(u)-g_2(v))}_{L^{q'}_tL^{r'}}\lesssim  t^{2\sigma/\theta} M^{2\sigma} \| \chi_R(u-v) \|_{L^q_t L^r_x}.
$$

$\circ$ \textit{Estimate of $e_1(t)$:} Thanks to Lemma \ref{i}, we get 
 \begin{align*}
        \norm{e_1}_{X_t}&\lesssim \norm{\chi_R(g_1(u)-g_1(v))}_{L^1_tL^2}\\
        &\lesssim \frac{4}{1-\alpha}\norm{\chi_R\abs{w}^{\alpha}}_{L^1_tL^2}\\
        &\lesssim \frac{4}{1-\alpha}\int_0^t\norm{\chi_R^{1-\alpha}\chi_R^{\alpha}\abs{w(z)}^{\alpha}}_{L^2}dz\\
        &\lesssim \frac{4}{1-\alpha} \abs{B_R}^{\frac{1-\alpha}{2}}\int_0^t\norm{\chi_Rw(z)}^{\alpha}_{L^2}dz.
    \end{align*}
Hence,
\begin{equation*}
    \| \chi_R w \|_{X_t} \le \| \chi_R w_0 \|_{L^2} + C \frac{tM}{R} 
+ C (t^{2\sigma/\theta} M^{2\sigma}+t)\norm{\chi_R w}_{X_t} 
+ \frac{4C}{1-\alpha} |B_R|^{\frac{1-\alpha}{2}} \int_0^t \| \chi_R w \|_{L^2}^\alpha.\qedhere
\end{equation*}
\end{proof}

\begin{proof}[Proof of Theorem \ref{08}]Let $T>0$ and denote $w_n = u_n - u$, and set $M := \max\bigl( \|v\|_{L^\infty_T H^1(\mathbb{R}^d)}, K_T \bigr).$
By Lemma \ref{uniq} we have 
\begin{equation}
    \| \chi_R w_n \|_{X_t} \le \| \chi_R (u_{n,0} - u_0) \|_{L^2} + C \frac{tM}{R} 
+ C (t^{2\sigma/\theta} M^{2\sigma}+t)\norm{\chi_R w}_{X_t}
+ \frac{4C}{1-\alpha} |B_R|^{\frac{1-\alpha}{2}} \int_0^t \| \chi_R w_n \|_{L^2}^\alpha.
\end{equation}
Choosing $\tau \in [0,T]$ such that $C (\tau^{2\sigma/\theta} M^{2\sigma}+\tau) \le \frac12$, for $t \in [0,\tau]$ we have
\begin{equation}
    \| \chi_R w_n(t) \|_{L^2} \le \| \chi_R (u_{n,0} - u_0) \|_{L^2} + C \frac{\tau M}{R} + \frac{4C}{1-\alpha} |B_R|^{\frac{1-\alpha}{2}} \int_0^t \| \chi_R w_n \|_{L^2}^\alpha.
\end{equation}
Taking $\alpha = 1 - \frac{1}{\log(R)}$, we have $|B_R|^{\frac{1-\alpha}{2}} \lesssim 1$, so
\begin{equation}\label{enfin}
    \| \chi_R w_n(t) \|_{L^2} \le \| \chi_R (u_{n,0} - u_0) \|_{L^2} + C_1 \frac{\tau M}{R} + 4 C_2 \log(R) \int_0^t \| \chi_R w_n \|_{L^2}^{1 - \frac{1}{\log(R)}}.
\end{equation}
Using Lemma~\ref{prs} with $\beta = \frac{1}{\log(R)}$, we deduce for $t \in [0,\tau]$:
\begin{equation}\label{enfin1}
    \| \chi_R w_n(t) \|_{L^2} \le \Biggl( \biggl( \frac{C(\tau,M)}{R} + \|\chi_R( u_{n,0} - u_0) \|_{L^2} \biggr)^{\frac{1}{\log(R)}} + C_3 \tau \Biggr)^{\log(R)}.
\end{equation}
Fix $R_0>0$ and take $R > R_0$, then
\begin{equation}\label{enfin12}
    \| w_n(t) \|_{L^2(B_{R_0})} \le \| \chi_R w_n(t) \|_{L^2} \le \Biggl( \biggl( \frac{C(\tau,M)}{R} + \|\chi_R( u_{n,0} - u_0) \|_{L^2} \biggr)^{\frac{1}{\log(R)}} + C_3 \tau \Biggr)^{\log(R)}.
\end{equation}
Since $u_{n,0} \to u_0$ in $L^2_{\rm loc}(\mathbb{R}^d)$, we obtain
\begin{equation}
\limsup_{n \to \infty} \| w_n(t) \|_{L^2(B_{R_0})} \le \Bigg( \bigg( \frac{C(\tau,M)}{R} \bigg)^{\frac{1}{\log(R)}} + C_3 \tau \Bigg)^{\log(R)}.
\end{equation}
Noting that $\bigl( \frac{C(\tau,M)}{R} \bigr)^{\frac{1}{\log(R)}} \to e^{-1}$ as $R \to \infty$, we choose $\tau$ such that $C_3 \tau \le \frac12 - e^{-1}$. Then
$$
\Biggl( \biggl( \frac{C(\tau,M)}{R} \biggr)^{\frac{1}{\log(R)}} + C_3 \tau \Biggr)^{\log(R)} \to 0 \quad \text{as } R \to \infty.
$$
We conclude $\limsup_{n\to+\infty}\norm{w_n(t)}_{L^2(B_{R_0})}=0,$
that is
$$
w_n(t) \to 0 \quad \text{in } L^2(B_{R_0}) \text{ for all } t \in [0,\tau].
$$
Since $R_0>0$ is arbitrary, we deduce
$$
w_n(t) \to 0 \quad \text{in } L^2_{\mathrm{\rm loc}}(\mathbb{R}^d) \text{ for all } t \in [0,\tau].
$$
Since $\tau$ depends only on $M$, by repeating the same arguments, we can pass from $\tau$ to $2\tau$, and then to $T$ by finite induction.
\end{proof}
\begin{corollary}
     Let $T > 0$, and let $u, v \in L^{\infty}(0,T; H^1(\mathbb{R}^d))$ be two solutions of equation~\eqref{NLS}
(in the sense of distributions). Then $u = v$.
\end{corollary}
\begin{proof}
 By taking $w_0=0$ in Lemma \ref{uniq}, and following the same proof as before, we obtain uniqueness.
\end{proof}
\subsection[L2 Critical case]{$L^2$--Critical case }\label{l2}
In this subsection, we consider the equation \eqref{NLSlp}. 
Let $\lambda \in \mathbb{R}^*$, $\mu>0$, and $\sigma=\frac{2}{d}$. 
Let $u_0 \in H^1(\mathbb{R}^d)$ be such that 
$$
\|u_0\|_{L^2} < \|Q\|_{L^2}.
$$
Let's again to consider the regularization equation in the critical case
\begin{equation}\label{eps}\tag{NLSlog$_\epsilon$P}
i\partial_t u_\varepsilon + \Delta u_\varepsilon + \lambda u_\varepsilon\log(\abs{u_\varepsilon}^2+\varepsilon)+\mu u_\varepsilon\abs{u_\varepsilon}^{4/d} = 0,
\qquad x \in \mathbb{R}^d,\ d\ge 1.
\end{equation}
There exist a maximal solution $u_\varepsilon\in \mathcal{C}([0,T_{\max}^\varepsilon),H^1(\mathbb{R}^d))\cap \mathcal{C}^1([0,T_{\max}^\varepsilon),H^{-1}(\mathbb{R}^d))$, and the mass and energy are conserved.  
Let us now prove that $T_{\max}^\varepsilon = +\infty$. 
Arguing as in the proof of Lemma~\ref{kpo}, we obtain that for all $t \in [0, T_{\max}^{\varepsilon})$,
 $$\frac{\mathrm{d}}{\mathrm{d}t}\left(\frac{1}{2}\norm{\nabla u_\varepsilon(t)}^2_2- \frac{d}{2(d+2)}\int \abs{u_\varepsilon(t)}^{\frac{4}{d}+2}\mathrm{d}x\right)\leq 3\abs{\lambda} \norm{\nabla u_\varepsilon(t)}_{L^2}^2.$$
Integrating in time, we obtain
$$\frac{1}{2}\norm{\nabla u_\varepsilon(t)}^2_2- \frac{d}{2(d+2)}\int \abs{u_\varepsilon(t)}^{\frac{4}{d}+2}\mathrm{d}x\leq 3\abs{\lambda} \int_0^t\norm{\nabla u_\varepsilon(s)}_{L^2}^2\mathrm{d}s+C(\norm{u_0}_{H^1}).$$
On the other hand, we already know that $$\frac{1}{\frac{4}{d}+2}\norm{u_\varepsilon(t)}_{\frac{4}{d}+2}^{\frac{4}{d}+2}\leq \frac12\bigg(\frac{\norm{u_0}_{L^2}}{\norm{Q}_{L^2}}\bigg)^{4/d}\norm{\nabla u_\varepsilon(t)}_{L^2}^2.$$
Then, $$\norm{\nabla u_\varepsilon(t)}_{L^2}^2\leq C_1\int_0^t \norm{\nabla u_\varepsilon(s)}^2_{L^2}\mathrm{d}s+C_2.$$
Thus, by Gronwall's lemma, we conclude that
$$
\sup_{\substack{0<\varepsilon<1}} \|u_\varepsilon(t)\|_{H^1} < +\infty.
$$
We proceed again by the same argument as in the sub-critical case and conclude that there exists a unique solution $u\in \mathcal{C}(\mathbb{R}_+,H^1(\mathbb{R}^d))\cap \mathcal{C}^1(\mathbb{R}_+,H^{-1}_{\rm loc}(\mathbb{R}^d))$ to \eqref{NLSlp}. Moreover, the mass is conserved along the flow.
\begin{remark}
We observe that we are always dealing with the global Cauchy theory, and we show that $T_{\max}^\varepsilon = +\infty$. In fact, the problem of the local Cauchy theory is not clear at all when the mass is greater than $\norm{Q}_2$ or when the nonlinearity is $L^2$--supercritical. To address this, one needs to ensure that $T_{\max}^\varepsilon \to T^* > 0$, which is not obvious.
\end{remark}
\section{Construction of the solution in \texorpdfstring{$W_1(\mathbb{R}^d)$}{W1}}\label{2em}
Throughout this section, for \eqref{NLS}, we assume that $\lambda \in \mathbb{R}^*$, that \ref{A1}--\ref{A2} are satisfied, and that either condition \ref{B2} holds with $0 < \sigma < \frac{2}{d-2}$, or condition \ref{B1} holds with $0 < \sigma < \frac{2}{d}$.
We also consider the mass-critical case $\sigma = \frac{2}{d}$ with $\|u_0\|_2 = \|Q\|_2$ and $\mu > 0$ for \eqref{NLSlp}.

\subsection{Bound and continuity in \texorpdfstring{$W_1(\mathbb{R}^d)$}{W1}}
Suppose that $u_0 \in W_1(\mathbb{R}^d)$. In this subsection, we show that the solution $u$ obtained belongs to $\mathcal{C}(\mathbb{R}_+, W_1(\mathbb{R}^d))$.
\begin{lemma}\label{rtt}
We have
    $$E_\varepsilon(u_0)\longrightarrow E_0(u_0).$$
\end{lemma}
\begin{proof}
    We know that for all $x > 0$, $\delta > 0$, and $0 < \varepsilon < 1$, we have
$$
|\log(x + \varepsilon)| \leq \frac{(x+1)^{\delta}}{\delta} + |\log(x)| \leq C(\delta)(x^{\delta}+1) + |\log(x)|.
$$
In particular, for $\delta + 2 \in (2, 2^*)$, we have
$$
\bigl| |u_0|^2 \log(|u_0|^2 + \varepsilon) \bigr| \leq C(\delta)(|u_0|^{\delta+2} + |u_0|^2) + \bigl| |u_0|^2 \log|u_0|^2 \bigr| \in L^1(\mathbb{R}^d).
$$
Moreover, we have
$$
\left| \int_0^{|u_0|} \frac{z^3}{z^2+\varepsilon}  \mathrm{d}z \right| \leq \frac{|u_0|^2}{2} \in L^1(\mathbb{R}^d),
$$
and therefore, by the dominated convergence theorem, we obtain
\begin{equation*}
    E_\varepsilon(u_0) \longrightarrow E_0(u_0) := E(u_0).\qedhere
\end{equation*}
\end{proof}
\begin{remark}
  The energy $E_\varepsilon(u)$ for $\varepsilon = 0$ is given by:
$$
E(u) := E_0(u) = \frac{1}{2} \int |\nabla u|^2  \mathrm{d}x - \frac{\lambda}{2} \int |u|^2 \log|u|^2  \mathrm{d}x + \frac{\lambda}{2} \int |u|^2  \mathrm{d}x -  \int G(|u|^2)\mathrm{d}x.
$$
We then have
$$
DE(u)(\phi) = \Re \int \nabla u \cdot \overline{\nabla \phi}  \mathrm{d}x - \lambda  \Re \int u \log|u|^2  \overline{\phi}  \mathrm{d}x -   \Re \int u g(|u|^{2})  \overline{\phi}  \mathrm{d}x.
$$
This shows that the equation \eqref{NLS} can be written in Hamiltonian form:
$$
i  \partial_t u = \partial E(u).
$$
\end{remark}

\begin{proposition}\label{u}
The function $u$, solution of equation \eqref{NLS} defined in Lemma \ref{ghj}, satisfies
$$u\in \mathcal{C}(\mathbb{R}_+, W_1(\mathbb{R}^d)).$$
In addition, if $\lambda>0$, we have
$$
u \in \mathcal{C}_b(\mathbb{R}_+, W_1(\mathbb{R}^d)),
$$
and the energy is conserved:
$$
E(u(t)) = E(u_0), \quad \text{for all } t \in \mathbb{R}_+.
$$
\end{proposition}

The main idea to prove this proposition comes from the work of Hayashi and Ozawa \cite{3}. For this, we introduce a function $\rho \in \mathcal{C}_c^{\infty}(\mathbb{C}, \mathbb{R})$ satisfying
\begin{equation}
\rho(z) =
\begin{cases}
1 & \text{if } |z| \leq \tfrac{1}{4}, \\
0 & \text{if } |z| \geq \tfrac{1}{2},
\end{cases}
\quad \text{and} \quad 0 \leq \rho(z) \leq 1 \text{ for all } z \in \mathbb{C}.
\end{equation}
Then, for $\varepsilon \in (0,1/2)$, we define the following functions:
\begin{align}
F_{1\varepsilon}(u) &= \rho(u) |u|^2 \log\big( |u|^2 + \varepsilon \big), \\
F_{2\varepsilon}(u) &= (1 - \rho(u)) |u|^2 \log\big(|u|^2 + \varepsilon \big),
\end{align}
and in the limit as $\varepsilon \to 0$, we obtain
\begin{align}
F_1(u) &= \rho(u) |u|^2 \log|u|^2, \\
F_2(u) &= (1 - \rho(u)) |u|^2 \log|u|^2.
\end{align}
In this case, our regularized energy is given by:
\begin{equation}\label{E}
E_\varepsilon(u) := \frac{1}{2} \int |\nabla u|^2  \mathrm{d}x - \frac{\lambda}{2} \int F_{1\varepsilon}(u)  \mathrm{d}x - \frac{\lambda}{2} \int F_{2\varepsilon}(u)  \mathrm{d}x + \lambda \int \eta_\varepsilon(|u|)  \mathrm{d}x - \int G(|u|^2)\mathrm{d}x.
\end{equation}
Using the elementary inequality from Lemma \ref{LN}, we deduce
$$
\int |F_2(u)|  \mathrm{d}x \leq \frac{1}{\delta} \int |u|^{2+2\delta}  \mathrm{d}x.
$$
Choosing $\delta > 0$ such that $2+2\delta \in (2, 2^*)$, we obtain, by the Sobolev embedding,
$$
\int |F_2(u)|  \mathrm{d}x \leq C(\delta) \| u \|_{H^1(\mathbb{R}^d)}^{2+2\delta}.
$$
Finally, since $u \in H^1(\mathbb{R}^d)$, we have
$$
u \in W_1(\mathbb{R}^d) \Longleftrightarrow F_1(u) \in L^1(\mathbb{R}^d).
$$
\begin{lemma}\label{gui}
For all $\alpha >0$, there exists $C(\alpha) > 0$ such that for all $z_1, z_2 \in \mathbb{C}$, we have
    \begin{align*}
         |F_{2\varepsilon}(z_1) - F_2(z_2)| &\leq 16\varepsilon \abs{z_1}^2+C(\alpha) (|z_1|^{1+\alpha} + |z_2|^{1+\alpha}) |z_1 - z_2|.
    \end{align*}
\end{lemma}
\begin{proof}
 We have
$$|F_{2\varepsilon}(z_1) - F_2(z_2)|=\abs{F_{2\varepsilon}(z_1)-F_2(z_1)}+\abs{F_2(z_1)-F_2(z_2)}\leq 16\varepsilon\abs{z_1}^2+\abs{F_2(z_1)-F_2(z_2)}.$$ 
For the first term, we have
$$
|F_{2\varepsilon}(z_1) - F_2(z_1)|=\abs{(1-\rho(z_1))\abs{z_1}^2\log(1+\frac{\varepsilon}{\abs{z_1}^2})} \leq (1-\rho(z_1)) \varepsilon \leq 16\varepsilon |z_1|^2.
$$
For the second term, we let $z(s):=z_2+s(z_2-z_1)$. Then we have $$F_2(z_1)-F_2(z_2)=\int_0^1\mathrm{d}F_2(z(s))(z_1-z_2)\mathrm{d}s.$$
and $$\mathrm{d}F_2(z(s))(z_1-z_2)=-\nabla\rho(z(s))\cdot\abs{z(s)}^2\log\abs{z(s)}^2(z_1-z_2)+2(1-\rho(z(s)))(\log\abs{z(s)}^2+1)\Re(\overline{z(s)}(z_1-z_2)).$$
Note that $$\operatorname{supp}\nabla\rho\subseteq\{z\in \mathbb{C}\mid \tfrac{1}{4}\leq \abs{z}\leq\tfrac12\}\quad \text{ and } \quad \operatorname{supp}(1-\rho)\subseteq \{z\in \mathbb{C}\mid  \abs{z}\geq\tfrac14\}.$$
Then, $$\abs{\nabla\rho(z(s))\cdot\abs{z(s)}^2\log\abs{z(s)}^2(z_1-z_2)}\lesssim_\alpha \abs{z(s)}^{1+\alpha}\abs{z_1-z_2}\lesssim \big(\abs{z_1}^{1+\alpha}+\abs{z_2}^{1+\alpha}\big)\abs{z_1-z_2},$$
and \begin{align*}
    \abs{(1-\rho(z(s)))\big(\log\abs{z(s)}^2+1\big)\Re(\overline{z(s)}(z_1-z_2))}&\leq 2\Big((1-\rho(z(s)))\big(\abs{z(s)\log\abs{z(s)}}+\abs{z(s)}\big)\Big)\abs{z_1-z_2}\\ &\lesssim_\alpha\abs{z(s)}^{1+\alpha}\abs{z_1-z_2}\\ &\lesssim_\alpha\Big(\abs{z_1}^{1+\alpha}+\abs{z_2}^{1+\alpha}\Big)\abs{z_1-z_2}.\qedhere
\end{align*}
\end{proof}
Before proving the proposition, we state the following lemma similarly as in \cite{3}:
\begin{lemma}\label{xc}
    Up to a subsequence, we have:
$$
\int \eta_{\varepsilon}(|u_\varepsilon(t)|)\mathrm{d}x \to \frac{1}{2} \int |u(t)|^2\mathrm{d}x , \quad 
\int F_{2\varepsilon}(u_\varepsilon(t)) \mathrm{d}x \to \int F_2(u(t))\mathrm{d}x ,
$$
for all $t \in \mathbb{R}_+$.

\end{lemma}
\begin{proof}
    \begin{itemize}
    \item Recall that 
    $$
    \eta_\varepsilon(|u_\varepsilon|) = \int_0^{|u_\varepsilon|} \frac{z^3}{z^2 + \varepsilon} \mathrm{d}z.
    $$
    For all $z_1, z_2 \in \mathbb{C}$, we have
    \begin{align*}
        \left| \int_0^{|z_2|} \frac{z^3}{z^2 + \varepsilon} \mathrm{d}z - \frac{1}{2}|z_1|^2 \right| 
        &\leq \left| \int_{|z_1|}^{|z_2|} \frac{z^3}{z^2 + \varepsilon} \mathrm{d}z \right| 
        + \left| \int_0^{|z_1|} \frac{z^3}{z^2 + \varepsilon} \mathrm{d}z - \frac{1}{2}|z_1|^2 \right| \\
        &\leq \frac{1}{2} \left| |z_2|^2 - |z_1|^2 \right| 
        + \left| \int_0^{|z_1|} \frac{z^3}{z^2 + \varepsilon} \mathrm{d}z - \frac{1}{2}|z_1|^2 \right| \\
        &\leq \frac{1}{2} \big(|z_2| + |z_1|\big) |z_2 - z_1| 
        + \left| \int_0^{|z_1|} \frac{z^3}{z^2 + \varepsilon} \mathrm{d}z - \frac{1}{2}|z_1|^2 \right|.
    \end{align*}
    Therefore, for all $t \in \mathbb{R}_+$,
    \begin{align*}
        \int \left| \int_0^{|u_\varepsilon(t)|} \frac{z^3}{z^2 + \varepsilon} \mathrm{d}z - \frac{1}{2} |u(t)|^2 \right|\mathrm{d}x 
        \lesssim_{\norm{u_0}_{L^2}} \| u_\varepsilon(t) - u(t) \|_{L^2} 
        + \int \left| \int_0^{|u(t)|} \frac{z^3}{z^2 + \varepsilon} \mathrm{d}z - \frac{1}{2} |u(t)|^2 \right|\mathrm{d}x.
    \end{align*}
    The first term tends to $0$ by the strong convergence $u_\varepsilon(t) \to u(t)$ in $L^2(\mathbb{R}^d)$. The second term tends to $0$ by the dominated convergence theorem.

    \item Taking $\alpha = \sigma$ in Lemma~\ref{gui}, and using the Cauchy--Schwarz inequality together with the Sobolev embedding $H^1(\mathbb{R}^d) \hookrightarrow L^{2\sigma+2}(\mathbb{R}^d)$, we obtain
    \begin{align*}
        \int |F_{2\varepsilon}(u_\varepsilon(t)) - F_2(u(t))|
        &\lesssim C(\alpha) (\|u_\varepsilon(t)\|_{2\sigma+2}^{\sigma+1} + \|u(t)\|_{2\sigma+2}^{\sigma+1}) \| u_\varepsilon(t) - u(t) \|_{L^2}+\varepsilon\norm{u_0}_2^2  \\
        &\lesssim C(N_T) \| u_\varepsilon(t) - u(t) \|_{L^2}+\varepsilon\norm{u_0}_2^2 .
    \end{align*}
\end{itemize}
We conclude by the strong convergence $u_\varepsilon(t) \to u(t)$ in $L^2(\mathbb{R}^d)$.
\end{proof}

\begin{proof}[Proof of Proposition \ref{u}]

$\bullet$ \textit{Case} $\lambda>0$: Let $T > 0$ and $t \in [0, T]$.  
From \eqref{E}, we deduce:
\begin{align*}
    \frac{1}{2} \| \nabla u_\varepsilon(t) \|_{L^2}^2 + \frac{\lambda}{2} \int |F_{1\varepsilon}(u_\varepsilon(t))| 
    &= E_\varepsilon(u_\varepsilon(t)) + \frac{\lambda}{2} \int F_{2\varepsilon}(u_\varepsilon(t)) - \lambda \int \eta_\varepsilon(|u_\varepsilon(t)|) +  \int G(|u_\varepsilon(t)|^2)\mathrm{d}x \\
    &= E_\varepsilon(u_0) + \frac{\lambda}{2} \int F_{2\varepsilon}(u_\varepsilon(t)) - \lambda \int \eta_\varepsilon(|u_\varepsilon(t)|) +  \int G(|u_\varepsilon(t)|^2)\mathrm{d}x.
\end{align*}
Again by Fatou's lemma, we deduce that,
$$
\frac{1}{2} \| \nabla u(t) \|_{L^2}^2 + \frac{\lambda}{2} \int |F_1(u(t))| \leq \liminf_{\varepsilon \to 0} \left( E_\varepsilon(u_0) + \frac{\lambda}{2} \int F_{2\varepsilon}(u_\varepsilon(t)) - \lambda \int \eta_\varepsilon(|u_\varepsilon(t)|) + \int G(|u_\varepsilon(t)|^2)\mathrm{d}x\right).
$$
Using Lemma \ref{xc} and Corollary \ref{rq2}, we get:
$$
\frac{1}{2} \| \nabla u(t) \|_{L^2}^2 + \frac{\lambda}{2} \int |F_1(u(t))| 
\leq E(u_0) + \frac{\lambda}{2} \int F_2(u(t)) - \frac{\lambda}{2} \int |u(t)|^2 + \int G(|u(t)|^2)\mathrm{d}x.
$$
Thus, 
$$
u(t) \in W_1(\mathbb{R}^d), \quad \text{and} \quad E(u(t)) \leq E(u_0), \quad \text{for all } t \in \mathbb{R}_+.
$$
By uniqueness, the opposite inequality also holds, and we get
$$
E(u(t)) = E(u_0), \quad \text{for all } t \in \mathbb{R}_+.
$$
Now, to show that $u \in L^\infty(\mathbb{R}_+, W_1(\mathbb{R}^d))$, we distinguish the cases according to the assumption \ref{B1} and \ref{B2}.

 We begin with the case where \ref{B2} holds. For all $t \in \mathbb{R}_+$,
\begin{align*}
    \frac{1}{2} \| \nabla u(t) \|_{L^2}^2 + \frac{\lambda}{2} \int |F_1(u(t))| 
    &\leq E(u_0) + \frac{\lambda}{2} \int F_2(u(t)) \\
    &\lesssim E(u_0) + \int |u(t)|^{2+\delta} \\
    &\lesssim E(u_0) + \| u_0 \|_{L^2}^{2+\delta - \frac{d\delta}{2}} \| \nabla u(t) \|_{L^2}^{\frac{d\delta}{2}}.
\end{align*}
The last inequality follows from the Gagliardo-Nirenberg inequality.
We choose $\delta > 0$ such that $\frac{d\delta}{2} < 2$, then by Lemma \ref{fait12}:
$$
\frac{1}{4} \| \nabla u(t) \|_{L^2}^2 + \frac{\lambda}{2} \int |F_1(u(t))| \lesssim E(u_0) + C(\norm{u_0}_{L^2}).
$$
 We then address the complementary case where \ref{B1} holds. For all $t \in \mathbb{R}_+$,
\begin{align*}
    \frac{1}{2} \| \nabla u(t) \|_{L^2}^2 + \frac{\lambda}{2} \int |F_1(u(t))| 
    &\leq E(u_0) + \frac{\lambda}{2} \int F_2(u(t)) + \int G(|u(t)|^2)\mathrm{d}x\\
    &\lesssim E(u_0) + \int |u(t)|^{2+\delta} +\int G(|u(t)|^2)\mathrm{d}x \\
    &\lesssim E(u_0) + \| u_0 \|_{L^2}^{2+\delta - \frac{d\delta}{2}} \| \nabla u(t) \|_{L^2}^{\frac{d\delta}{2}} +  \| u_0 \|_{L^2}^{2\sigma+2 - d\sigma} \| \nabla u(t) \|_{L^2}^{d\sigma}.
\end{align*}
Again, the last inequality follows from the Gagliardo-Nirenberg inequality. We choose $\delta > 0$ such that $\frac{d\delta}{2} < 2$, and since $d\sigma < 2$, then by Lemma \ref{fait12}:
$$
\frac{1}{4} \| \nabla u(t) \|_{L^2}^2 + \frac{\lambda}{2} \int |F_1(u(t))| \lesssim E(u_0) + C(\norm{u_0}_{L^2}).
$$
Finally, in both cases, we conclude that 
$$
u \in L^\infty(\mathbb{R}_+, W_1(\mathbb{R}^d)).
$$

$\bullet$ \text{Case} $\lambda<0$: First we prove that $u\in L_{\rm loc}^{\infty}(\mathbb{R}_+,W_1(\mathbb{R}^d))$, from \eqref{E}, we obtain
\begin{align*}
    \left|\frac{\lambda}{2}\right| \int |F_{1\varepsilon}(u_\varepsilon(t))| &= \frac{\lambda}{2} \int F_{1\varepsilon}(u_\varepsilon(t)) \\
    &= -E_\varepsilon(u_\varepsilon(t)) + \frac{1}{2} \| \nabla u_\varepsilon(t) \|_{L^2}^2 - \frac{\lambda}{2} \int F_{2\varepsilon}(u_\varepsilon(t)) + \lambda \int \eta_\varepsilon(|u_\varepsilon(t)|) -\int G(|u_\varepsilon(t)|^2) \\
    &\leq -E_\varepsilon(u_0) + C(N_T).
\end{align*}
Then, by Fatou's lemma and Lemma \ref{rtt}, we deduce that
$$
\frac{|\lambda|}{2} \int |F_1(u(t))| \leq -E(u_0) + C(N_T),
$$
and therefore the mapping $t \mapsto \int |u(t)|^2 |\log|u(t)|^2|$ belongs to $L_{\text{loc}}^{\infty}(\mathbb{R}_+)$, in particular, 
$$
u \in L_{\text{loc}}^{\infty}(\mathbb{R}_+, W_1(\mathbb{R}^d)).
$$
Now we claim that $u\in\mathcal{C}(\mathbb{R}_+,W_1(\mathbb{R}^d))$, Thanks to Lemma \ref{ii}, we have
$
t \mapsto \int F_2(u(t)) \in \mathcal{C}(\mathbb{R}_+).
$ Hence by Lemma \ref{continue} we need to check that $t\mapsto \int F_1(u(t))\in \mathcal{C}(\mathbb{R}_+)$. To this end, by uniqueness it suffices to establish the continuity of the function at $0$.
\\
By Lemma \ref{jamila} and Lemma \ref{1}, we have $$\abs{\int \abs{u_\varepsilon(t)}^2\log(\abs{u_\varepsilon(t)}^2+\varepsilon)\mathrm{d}x-\int \abs{u_0}^2\log(\abs{u_0}^2+\varepsilon)\mathrm{d}x}\leq C\int_0^t\norm{\nabla u_\varepsilon(s)}^2_2\mathrm{d}s\leq C(e^{C t}-1).$$
Hence, $$\int \abs{F_{1\varepsilon}(u_\varepsilon(t))}\leq C(e^{C t}-1)-\int \abs{u_0}^2\log(\abs{u_0}^2+\varepsilon)+\int F_{2\varepsilon}(u_\varepsilon(t)).$$
Applying Fatou’s lemma, together with Lemma \ref{xc}, we obtain $$\int \abs{F_1(u(t))}\leq\liminf_{\varepsilon\to 0}\int \abs{F_{1\varepsilon}(u_\varepsilon(t))}\leq C(e^{C t}-1)-\int \abs{u_0}^2\log\abs{u_0}^2+\int F_{2}(u(t)).$$
It follows that $$\limsup_{t\to 0}\int \abs{F_1(u(t))}\mathrm{d}x\leq \int |F_1(u_0)|\mathrm{d}x.$$
On the other hand, by Fatou’s lemma, the map $t\mapsto \int \abs{F_1(u(t))}\mathrm{d}x$ is lower semicontinuous. Therefore, it is continuous at $t=0$.
\end{proof}
\begin{proposition}\label{op}
   The operator $L: u \mapsto \Delta u + \lambda u \log|u|^2 +  u g(|u|^{2})$ is continuous and bounded from $W_1(\mathbb{R}^d)$ to $W_1'(\mathbb{R}^d)$.

 Moreover, the function u defined in Lemma \ref{ghj} satisfies
$
u \in \mathcal{C}(\mathbb{R}_+, W_1(\mathbb{R}^d))
\cap \mathcal{C}^1(\mathbb{R}_+, W_1'(\mathbb{R}^d)),
$
and 
\begin{equation}\label{equation2}
i u_t + \Delta u + \lambda u \log|u|^2 + u g(|u|^{2})
= 0 \quad \text{in } W_1'(\mathbb{R}^d), \quad \text{for all } t \in \mathbb{R}_+.
\end{equation}
\end{proposition}
Before proving the proposition, we recall the key lemma. 
\begin{lemma}\cite[Lemma 2.6]{7}\label{oper}
    The operator 
$$
L: u \in W_1(\mathbb{R}^d) \longmapsto u \log|u|^2 \in W_1'(\mathbb{R}^d),
$$
is continuous and bounded.
\end{lemma}
\begin{proof}[Proof of Proposition \ref{op}]
    By Lemma $\ref{oper}$, it suffices to prove the statement of the proposition for the mapping $u \mapsto u g(|u|^{2})$. Let $u, v \in H^1(\mathbb{R}^d)$. Using Corollary \ref{power10} and Sobolev embeddings, we have
\begin{align*}
    \| u g(|u|^{2}) - v g(|v|^{2}) \|_{H^{-1}} 
&\lesssim \| u |u|^{2\sigma+2} - v |v|^{2\sigma+2} \|_{L^{\frac{2\sigma+2}{2\sigma+1}}}+\norm{u-v}_{L^2}\\
&\lesssim  \left( \| u \|_{L^{2\sigma+2}}^{2\sigma} + \| v \|_{L^{2\sigma+2}}^{2\sigma} \right) \| u - v \|_{L^{2\sigma+2}}+\norm{u-v}_{L^2}
\\&\lesssim \left( 1+\| u \|_{H^1(\mathbb{R}^d)}^{2\sigma} + \| v \|_{H^1(\mathbb{R}^d)}^{2\sigma} \right) \| u - v \|_{H^1(\mathbb{R}^d)},
\end{align*}
which shows that $u \mapsto u g(|u|^{2})$ is continuous from $H^1(\mathbb{R}^d)$ to $H^{-1}(\mathbb{R}^d)$.

Now from equation \eqref{equation}, we have
$$
i u_t = -L u,
$$
We deduce that $u \in \mathcal{C}^1(\mathbb{R}_+, W_1'(\mathbb{R}^d))$. In particular, we have
\begin{equation*}
    i u_t + \Delta u + \lambda u \log|u|^2 + u g(|u|^{2}) = 0 \quad \text{in } W_1'(\mathbb{R}^d), \quad \text{for all } t \in \mathbb{R}_+.\qedhere
\end{equation*}
\end{proof}
\section{Construction of the solution in \texorpdfstring{$\Sigma_\alpha$}{sigmaalpha}}\label{3em}
Throughout this section, for \eqref{NLS}, we assume that $\lambda \in \mathbb{R}^*$, that \ref{A1}--\ref{A2} are satisfied, and that either condition \ref{B2} holds with $0 < \sigma < \frac{2}{d-2}$, or condition \ref{B1} holds with $0 < \sigma < \frac{2}{d}$.
We also consider the mass-critical case $\sigma = \frac{2}{d}$ with $\|u_0\|_2 = \|Q\|_2$ and $\mu > 0$ for \eqref{NLSlp}.
\subsection{Energy conservation and continuity in \texorpdfstring{$\Sigma_\alpha$}{Sigmaalpha}}
In~\cite{13}, the authors study the construction of solutions $u \in L_{\rm loc}^{\infty}(\mathbb{R},\Sigma_\alpha)$ to \eqref{NLSlp} in a non-detailed manner. Moreover, they construct solutions in the weak sense under the condition $\mu < 0$, but their construction is equally valid for $\mu > 0$.
In this section, we provide a detailed construction of solutions $u\in \mathcal{C}(\mathbb{R},\Sigma_\alpha)$ and we prove the conservation of energy without any sign condition on $\lambda$.
Assume that $u_0 \in \Sigma_\alpha$. Let $u$ and $u_\varepsilon$ be the solution to \eqref{NLS} and $\eqref{ppe}$ respectively as defined previously. We know that $u_\varepsilon$ is uniformly bounded in $H^1(\mathbb{R}^d)$.

To show that $u$ is in $\Sigma_\alpha$, it suffices to show that the quantity
$
\|\langle x \rangle^\alpha u_\varepsilon(t)\|_{L^2}
$
is uniformly bounded with respect to $\varepsilon$.
\begin{remark}
   Note that one can also prove directly that $u \in \Sigma_\alpha$ by using the construction from the previous section and the technique developed in \cite[Section 6.5]{1}. 
However, the main objective here is to show that the energy is conserved for any sign of $\lambda$. For this purpose, we proceed via the regularized equation.
\end{remark}
We define $I_\alpha(v) := \|\langle x \rangle^\alpha v\|_{L^2}^2$, and we have the following estimate similarly as in \cite{13}:

\begin{lemma}\label{momo}
   Let $u_\varepsilon$ be the solution of the equation \eqref{ppe}, defined in Theorem \ref{ppp}. Then 
$t \mapsto I_\alpha(u_\varepsilon(t)) \in \mathcal{C}^1([0,T])$.  
Moreover, there exists a constant $C(T, \norm{u_0}_{\Sigma_\alpha}) > 0$ such that
\begin{equation}\label{C2}
    \|\langle x \rangle^\alpha u_\varepsilon(t)\|_{L^2} \leq C(T, \norm{u_0}_{\Sigma_\alpha}).
\end{equation}
\end{lemma}
\begin{proof}
For the regularity of the function $t \mapsto I_\alpha(u_\varepsilon(t))$, we follow the same technique developed in \cite[Section 6.5]{1}. 
Similarly as in \cite{13}, we have
\begin{align*}
\frac{1}{2} \frac{\mathrm{d}}{\mathrm{d}t} I_\alpha(u_\varepsilon(t))
= 2\alpha  \Im\left( \langle x \rangle^{2\alpha - 1} \nabla u_\varepsilon(t), u_\varepsilon(t) \right).
\end{align*}
On the other hand, since $\alpha \leq 1$, we have $\langle x \rangle^{2\alpha - 1} \leq \langle x \rangle^{\alpha}$. Moreover, from (\ref{nt}), we obtain
$$
\frac{1}{2} \frac{\mathrm{d}}{\mathrm{d}t} I_\alpha(u_\varepsilon(t)) \leq 2 \alpha \| \langle x \rangle^\alpha u_\varepsilon(t) \|_{L^2} N_T
= 2 \alpha \sqrt{I_\alpha(u_\varepsilon(t))}  N_T.
$$
Integrating this inequality in time from $0$ to $t$, we deduce
\begin{equation*}
    \|\langle x\rangle^{\alpha} u_\varepsilon(t)\|_{L^2} \leq 2\alpha \widetilde{N}_T + \|\langle x\rangle^{\alpha} u_0\|_{L^2} 
    =: C(T, \norm{u_0}_{\Sigma_\alpha}).\qedhere
\end{equation*}
\end{proof}
\begin{corollary}  Let $u$ be the solution of equation \eqref{NLS} given by Proposition \ref{op}. Then
$$
t \longmapsto \int_{\mathbb{R}^d} |\langle x \rangle^{\alpha} u(t,x)|^2  \mathrm{d}x \in L^{\infty}_{\mathrm{loc}}(\mathbb{R}_+).
$$
\end{corollary}
\begin{proof}
   Indeed, since the sequence $(\langle x \rangle^{\alpha} u_\varepsilon(t))_\varepsilon$ is uniformly bounded in $L^2(\mathbb{R}^d)$ for all $t \in [0,T]$, we deduce, by the same arguments as in the previous section, that
$$
\langle x \rangle^{\alpha} u_\varepsilon(t) \rightharpoonup \langle x \rangle^{\alpha} u(t) \quad \text{weakly in } L^2(\mathbb{R}^d), \quad \text{for all } t \in [0,T].
$$
In particular, from \eqref{C2}, for all $t \in [0, T]$, we have
$$
\|\langle x \rangle^{\alpha} u(t)\|_{L^2} \leq \liminf_{\varepsilon \to 0} \|\langle x \rangle^{\alpha} u_\varepsilon(t)\|_{L^2} \leq C(T, \norm{u_0}_{\Sigma_\alpha}).
$$
This shows that
$
t \longmapsto \|\langle x \rangle^{\alpha} u(t)\|_{L^2} \in L^{\infty}_{\mathrm{loc}}(\mathbb{R}_+).
$
\end{proof}
In the previous section, we do not know if we have conservation of the energy in the case $\lambda<0$.
By contrast, in the present setting we show that the energy is conserved even when $ \lambda < 0 $.
This improvement is due to the gain of regularity on $ u $.
The main difficulty encountered in the $W_1(\mathbb{R}^d)$ framework when $\lambda<0$ comes from the opposite signs of the first two terms in the regularized energy $E_\varepsilon(u_\varepsilon)$,
$$
\frac{1}{2}\int |\nabla u_\varepsilon|^2 
- \frac{|\lambda|}{2}\int |F_{1,\varepsilon}(u_\varepsilon)|,
$$
together with the fact that the convergence of these two terms is not known. 
However, with the uniform bound on $u_\varepsilon(t)$ in $ \Sigma_\alpha $, we can establish the following:
\begin{lemma}\label{xfg}
We have
    $$
\int F_{1\varepsilon}(u_\varepsilon(t))\mathrm{d}x \longrightarrow \int F_1(u(t))\mathrm{d}x \quad \text{as } \varepsilon \to 0, \text{ for all } t \in \mathbb{R}_+.
$$
\end{lemma}
\begin{proof}
This limit is achieved by splitting the integral into the regions $ |x| < R $ and $ |x| > R $,
applying Hölder's inequality, and optimizing with respect to $ R $ (see \cite{8}), techniques also used in \cite{13}.

For $R>0$ we have
\begin{align*}
\int \big|F_{1\varepsilon}(u_\varepsilon(t)) - F_1(u(t))\big| 
&= \int_{|x|<R} \big|F_{1\varepsilon}(u_\varepsilon(t)) - F_1(u(t))\big| 
+ \int_{|x|\ge R} \big|F_{1\varepsilon}(u_\varepsilon(t)) - F_1(u(t))\big|.
\end{align*}
$\bullet$ \textit{Second integral:} 
By Lemma \ref{LN} we have
$$
\int_{|x|\ge R} \big|F_{1\varepsilon}(u_\varepsilon(t)) - F_1(u(t))\big| 
\lesssim \int_{|x|>R} |u_\varepsilon(t)|^{2-2\delta} + |u(t)|^{2-2\delta}
$$
Applying H\"older's inequality with $p = \frac{1}{1-\delta}$ and $p' = \frac{1}{\delta}$, we obtain
\begin{align*}
    \int_{|x|>R} |u|^{2-2\delta} 
    &\leq \left( \int_{|x|>R} |x|^{-\alpha(2-2\delta)\frac{1}{\delta}}  \mathrm{d}x \right)^{\delta} 
          \cdot \left( \int_{|x|>R} |x^\alpha u(x)|^2  \mathrm{d}x \right)^{1-\delta}.
\end{align*}
If
$
0 < \delta < \frac{2\alpha}{d + 2\alpha},
$
then
$$
\left( \int_{|x|>R} |x|^{-\alpha(2-2\delta)\frac{1}{\delta}}  \mathrm{d}x \right)^\delta \leq c_2 R^{-2\alpha(1-\delta) + d\delta}.
$$
Hence,
\begin{equation}\label{zzz}
    \int_{|x|>R} |u|^{2-2\delta} \lesssim R^{-2\alpha(1-\delta) + d\delta}  \| x^\alpha u \|_{L^2}^{2-2\delta}.
\end{equation}
Same for $u_\varepsilon$, then by Lemma \ref{momo}, we get $$
\int_{|x|\ge R} \big|F_{1\varepsilon}(u_\varepsilon(t)) - F_1(u(t))\big|  
\lesssim_{T,\norm{u_0}_{\Sigma_\alpha}} \frac{1}{R^b},
$$
where $b = 2\alpha(1-\delta) - d\delta$ with $\delta \in (0,1)$ chosen such that $b>0$.\\
Let $\eta > 0$. Then there exists $R_0>0$ independent of $\varepsilon$ such that
$$
\int_{|x| \ge R_0} \big|F_{1\varepsilon}(u_\varepsilon(t)) - F_1(u(t))\big| < \frac{\eta}{2}.
$$
$\bullet$ \textit{First integral:} For this choice of $R_0$, by Lemma \ref{1.2} we have
$$
\lim_{\varepsilon \to 0} \int_{|x| < R_0} \big|F_{1\varepsilon}(u_\varepsilon(t)) - F_1(u(t))\big| = 0.
$$
Hence, there exists $\varepsilon_0 > 0$ such that for all $\varepsilon < \varepsilon_0$,
$$
\int_{|x| < R_0} \big|F_{1\varepsilon}(u_\varepsilon(t)) - F_1(u(t))\big| < \frac{\eta}{2}.
$$
$\bullet$ \textit{Conclusion:} Combining the two estimates, for $\varepsilon < \varepsilon_0$ we obtain
$$
\int \big|F_{1\varepsilon}(u_\varepsilon(t)) - F_1(u(t))\big| < \eta.
$$
This shows that
\begin{equation*}\label{cv12}
\int F_{1\varepsilon}(u_\varepsilon(t)) \longrightarrow \int F_1(u(t)), \quad \text{for all } t \in \mathbb{R}_+.\qedhere
\end{equation*}
\end{proof}
\begin{lemma}
    Let $u$ be the solution of equation \eqref{NLS} given by Proposition \ref{op}. Then
$
u \in \mathcal{C}(\mathbb{R}_+, \Sigma_\alpha),
$
and the energy is conserved:
$$
E(u(t)) = E(u_0), \quad \text{for all } t \in \mathbb{R}_+.
$$
\end{lemma}
\begin{proof}
 We prove first the conservation of energy 
  
If $\lambda > 0$, since $\Sigma_\alpha \subset W_1(\mathbb{R}^d)$, the energy is conserved.

If $\lambda < 0$, we have 
$$
E_\varepsilon(u_\varepsilon(t)) = E_\varepsilon(u_0).
$$
Hence,
$$
\frac{1}{2} \| \nabla u_\varepsilon(t) \|_{L^2}^2 = E_\varepsilon(u_0) + \frac{\lambda}{2} \int F_{1\varepsilon}(u_\varepsilon(t)) + \frac{\lambda}{2} \int F_{2\varepsilon}(u_\varepsilon(t)) - \lambda \int \eta_\varepsilon(|u_\varepsilon(t)|) + \int G(|u_\varepsilon(t)|^2).
$$
Using Lemma $\ref{xc}$, Corollary $\ref{rq2}$, Lemma \ref{xfg}, and Lemma $\ref{1.2}$, we obtain
$$
\frac{1}{2} \| \nabla u(t) \|_{L^2}^2 \le E(u_0) + \frac{\lambda}{2} \int F_1(u(t)) + \frac{\lambda}{2} \int F_2(u(t)) - \frac{\lambda}{2} \int |u(t)|^2 +  \int G(|u(t)|^2).
$$
In particular,
$$
E(u(t)) \le E(u_0), \quad \text{for all } t \in \mathbb{R}_+.
$$
and we get the equality by uniqueness.\\
Let us now show that $u \in \mathcal{C}(\mathbb{R}_+, \Sigma_\alpha)$. By Proposition \ref{u}, it suffices to show that $t\mapsto \norm{\abs{x}^{\alpha}u(t)}_{L^2}\in \mathcal{C}(\mathbb{R}_+).$ which can be proved in the same way as in \cite[Section 6.5]{1}, we conclude
$
u \in \mathcal{C}(\mathbb{R}_+, \Sigma_\alpha).
$
\end{proof}

\begin{lemma}
     The operator $L: u \mapsto \Delta u + \lambda u \log(|u|^2) +  u g(|u|^{2})$ is continuous and bounded from $\Sigma_\alpha(\mathbb{R}^d)$ to $\Sigma_\alpha'(\mathbb{R}^d)$.

 Moreover, the function $u$ defined in Lemma \ref{ghj} satisfies
$
u \in \mathcal{C}(\mathbb{R}_+, \Sigma_\alpha(\mathbb{R}^d))
\cap \mathcal{C}^1(\mathbb{R}_+, \Sigma_\alpha'(\mathbb{R}^d)),
$
and 
\begin{equation}
i u_t + \Delta u + \lambda u \log|u|^2 + u g(|u|^{2})
= 0 \quad \text{in } \Sigma_\alpha'(\mathbb{R}^d), \quad \text{for all } t \in \mathbb{R}_+.
\end{equation}
\end{lemma}
\begin{proof}
Thanks to \eqref{equation2}, we have
$$
i  \partial_t u = - L u.
$$
where $Lu=\Delta u+\lambda u\log(\abs{u}^2)+ ug(\abs{u}^{2})$. On the other hand, by diagram \eqref{dig} and the Proposition \ref{op}, the operator $L$ is continuous and bounded from $\Sigma_\alpha$ to $\Sigma_\alpha'$. 
Therefore, this proves that
$
u \in \mathcal{C}^1(\mathbb{R}_+, \Sigma_\alpha').
$
\end{proof}

\section*{Acknowledgments}
The authors acknowledge the support of the CDP C2EMPI, as well as the French State under the France-2030 programme, the University of Lille, the Initiative of Excellence of the University of Lille, and the European Metropolis of Lille for their funding and support of the R-CDP-24-004-C2EMPI project.
I would like to express my deepest gratitude to my supervisors, Vianney Combet, Guillaume Ferriere, and Sahbi Keraani, for their invaluable guidance, availability, and support throughout the completion of this work.
\renewcommand{\appendixpagename}{Appendix}
\renewcommand{\appendixtocname}{Appendix}

\appendix
\appendixpage
\addappheadtotoc
\setcounter{theorem}{0}
\section{Logarithmic inequalities}
\begin{lemma}\label{LN}
Let $c>0$, then for all $x \ge c$ and all $\delta > 0$, we have
$$
\log(x) \le \frac{1}{\delta c^{\delta}} x^{\delta}+\log(c).
$$
Moreover, for all $x > 0$ and all $\delta_1, \delta_2 > 0$, 
$$
|\log(x)| \le \max\Big(\frac{1}{\delta_1}, \frac{1}{\delta_2}\Big) \left( x^{\delta_1} + x^{-\delta_2} \right).
$$
\end{lemma}
\begin{lemma}\cite[Lemme 1.1.1]{29}\label{ingln}
For all $z_1, z_2 \in \mathbb{C}$, we have
$$
\left| \Im\Big(( z_1 \log|z_1| - z_2 \log|z_2| ) \overline{(z_1 - z_2)} \Big) \right| \le |z_1 - z_2|^2.
$$
\end{lemma}





\begin{lemma}\cite[Lemma A.2]{3}\label{cv2}
For all $\varepsilon,\delta_1, \delta_2 \in (0,1)$ and all $z_1, z_2 \in \mathbb{C}$, we have
\begin{align*}
| z_2 \log(|z_2|+\varepsilon) - z_1 \log|z_1| | 
&\lesssim_{\delta_1, \delta_2} \varepsilon + |z_1 - z_2| \\
&\quad + |z_1 - z_2|^{1/2} \Big( |z_2|^{1/2 - \delta_1} + |z_2|^{1/2 + \delta_2}  + |z_1|^{1/2 - \delta_1} + |z_1|^{1/2 + \delta_2} \Big).
\end{align*}
\end{lemma}

\begin{lemma}\cite[Lemma 3.1]{14}\label{i}
For all $(\alpha,c) \in (0,1)\times(0,1]$ and all $|z_1|, |z_2| \leq c$, we have
$$
| z_1 \log|z_1| - z_2 \log|z_2| | \le \frac{4}{1-\alpha} |z_1 - z_2|^{\alpha}.
$$
\end{lemma}



\begin{lemma}\cite[Lemma 3.1]{14}\label{ii}
For all $\alpha,c > 0$ and all $|z_1|, |z_2| \ge c$, we have
$$
| z_1 \log|z_1| - z_2 \log|z_2| | \le C(\alpha,c) \big( |z_1|^{\alpha} + |z_2|^{\alpha} \big) |z_1 - z_2|.
$$
\end{lemma}

\section{Some basic Lemmas}
In this section, we recall some basic lemmas that will be used throughout the paper.
\begin{lemma}\cite[p.84]{1}\label{power}
For all $z_1, z_2 \in \mathbb{C}$ and $\alpha \ge 0$, we have
$$
| z_1 |z_1|^{\alpha} - z_2 |z_2|^{\alpha} | \le (\alpha + 1) \big( |z_1|^{\alpha} + |z_2|^{\alpha} \big) |z_1 - z_2|.
$$
\end{lemma}
\begin{corollary}\label{power10}
    For all $z_1, z_2 \in \mathbb{C}$ and $\alpha \ge 0$, we have $$
| z_1 g(|z_1|^2) - z_2 g(|z_2|^2) | \lesssim \big( 1+|z_1|^{2\sigma} + |z_2|^{2\sigma} \big) |z_1 - z_2|,
$$ 
and 
$$\abs{G(|z_1|^2)-G(|z_2|^2)}\lesssim \abs{\abs{z_1}^2-\abs{z_2}^2}+\abs{\abs{z_1}^{2\sigma+2}-\abs{z_2}^{2\sigma+2}}.$$
\end{corollary}
\begin{proposition}[Gagliardo-Nirenberg \cite{6}]\label{GN}
Let $0\leq p\leq \frac{2}{d-2}$ and $\theta=\frac{dp}{2p+2}$. Then, for all $u\in H^1(\mathbb{R}^d)$, 
    \begin{equation}\label{gn}
        \norm{u}_{L^{2p+2}}\leq C_{GN}\norm{\nabla u}_{L^2}^{\theta}\norm{u}_{L^2}^{1-\theta}.
    \end{equation}
\end{proposition}
\begin{lemma}\label{x11}
Let $a,b,A,B \in \mathbb{R}_+^*$. Then the function 
$$
f(x) = A x^a + B x^{-b}, \quad x \in \mathbb{R}_+^*,
$$
admits a unique minimizer, which is attained at 
$$
x = \left(\frac{bB}{aA}\right)^{\frac{1}{a+b}}.
$$
\end{lemma}

\begin{lemma}\cite[Lemma 3.4]{14}\label{prs}
Let $T>0$, $a,b \ge 0$, $\beta \in (0,1)$, and let $f \in L^1(0,T)$ be a positive function. Suppose that
$$
f(t) \le a + \frac{b}{\beta} \int_0^t f(s)^{1-\beta}  ds, \quad \forall t \in [0,T].
$$
Then, for all $t \in [0,T]$, we have
$$
f(t) \le (a^{\beta} + b t)^{1/\beta}.
$$
\end{lemma}
\begin{lemma}\label{fait12}
        Let $x>0$, and let $P(x)=a_1 x^{s_1} + \cdots + a_n x^{s_n}$, where $s_i \in (0,1)$ and $a_i\geq0$ for all $i$. Then there exists a constant $c\geq0$ such that
$$
x - P(x) \ge \frac{1}{2}x - c.
$$
    \end{lemma} 

\begin{lemma}[Aubin-Lions \cite{4}]\label{aubin}
Let $X_0$, $X$, and $X_1$ be three Banach spaces such that:

\begin{itemize}
    \item $X_0 \hookrightarrow X$ with a compact embedding,
    \item $X \hookrightarrow X_1$ with a continuous embedding.
\end{itemize}

For $p, q \in [1, +\infty]$, define
$$
\Gamma := \left\{ u \in L^p(0,T; X_0) \ \middle| \ \partial_t u \in L^q(0,T; X_1) \right\}.
$$

Then:
\begin{itemize}
    \item[(i)] If $p < \infty$, the embedding of $\Gamma$ into $L^p(0,T; X)$ is compact.
    \item[(ii)] If $p = \infty$ and $q > 1$, the embedding of $\Gamma$ into $\mathcal{C}([0,T]; X)$ is compact.
\end{itemize}
\end{lemma}
\section[Spaces Sigma and W1]{Spaces $\Sigma_\alpha(\mathbb{R}^d)$ and $W_1(\mathbb{R}^d)$}\label{C}
The Orlicz space defined$$
W_1(\mathbb{R}^d) := \big\{ u \in H^1(\mathbb{R}^d) \ \big| \ |u|^2 \log|u| \in L^1(\mathbb{R}^d) \big\},
$$
is a Banach space endowed with the Luxemburg norm defined by
$$
\norm{u}_{W_1(\mathbb{R}^d)} = \norm{u}_{H^1(\mathbb{R}^N)} + \inf \Bigg\{ k > 0 \ \Big| \ \int_{\mathbb{R}^d} A\big(k^{-1}|u|\big) \mathrm{d}x \le 1 \Bigg\},
$$
where
$$
A(s) =
\begin{cases}
-s^2 \log s^2, & \text{if } 0 \le s \le e^{-3},\\
3 s^2 + 4 e^{-3} s - e^{-6}, & \text{if } s \ge e^{-3}.
\end{cases}
$$
For $0 < \alpha \le 1$, we define the space
$$
\Sigma_\alpha := \Bigl\{ u \in H^1(\mathbb{R}^d) \Big| \langle x \rangle^{\alpha} u \in L^2(\mathbb{R}^d) \Bigr\},
$$
where for $x \in \mathbb{R}^d$, 
$$
\langle x \rangle := \sqrt{1 + |x|^2}.
$$
We endow $\Sigma_\alpha$ with the norm
$$
\norm{u}_{\Sigma_\alpha}^2 := \norm{u}_{H^1(\mathbb{R}^d)}^2 + \|\langle x \rangle^{\alpha} u\|_{L^2(\mathbb{R}^d)}^2.
$$
Equipped with this norm, $\Sigma_\alpha$ is a reflexive and separable Banach space.
we remark that the following diagram commutes and all the inclusions are continuous:  
\begin{equation}\label{dig}
    \begin{tikzcd}
\Sigma_\alpha \arrow[r, hook] \arrow[d, hook] & W_1(\mathbb{R}^d) \arrow[r, hook] \arrow[d, hook] & H^1(\mathbb{R}^d) \arrow[r, hook] \arrow[d, hook] & L^2(\mathbb{R}^d) \arrow[d, hook] \\
\Sigma_\alpha' & \arrow[l, hook'] W_1(\mathbb{R}^d)' & \arrow[l, hook'] H^{-1}(\mathbb{R}^d) & \arrow[l, hook'] L^2(\mathbb{R}^d)
\end{tikzcd}
\end{equation}
\section*{Declaration of competing interest}
No competing interest.

\end{document}